\documentclass{amsart}
\usepackage{amsmath}
\usepackage{amsfonts}
\usepackage{amssymb}
\usepackage{graphicx}%
\usepackage{enumerate}
\newtheorem{definition}{Definition}[section]
\newtheorem{remark}{Remark}[section]
\newtheorem{theorem}{Theorem}[section]
\newtheorem{lemma}{Lemma}[section]
\newtheorem{proposition}{Proposition}[section]

\usepackage{graphicx}
\usepackage{color}
\usepackage[dvipsnames]{xcolor} 
\makeatletter
\renewcommand{\@makefnmark}{\hbox{\@textsuperscript{\normalfont\@thefnmark}}}
\makeatother
\usepackage[colorlinks = true,
            linkcolor  = blue,
            citecolor  = blue,
            urlcolor   = blue,
            bookmarks  = true]{hyperref}

\theoremstyle{remark}

\numberwithin{equation}{section}

\usepackage[backend=biber, style=numeric, citestyle=numeric-comp]{biblatex}
\DeclareFieldFormat{labelnumberwidth}{\mkbibbrackets{#1}}
\DeclareNameAlias{author}{family-given}

\DeclareFieldFormat*{title}{#1}

\AtEveryBibitem{%
  \clearfield{issn}
  \clearfield{url}
  \clearfield{eprint}
}

\renewbibmacro{in:}{}

\DeclareFieldFormat[article]{volume}{\textbf{#1}}
\DeclareFieldFormat[article]{number}{\textbf{#1}}
\renewbibmacro*{volume+number+eid}{%
  \printfield{volume}
  \iffieldundef{number}
    {}%
    {\printtext[parens]{\printfield{number}}}
  \setunit{\addcomma\space}%
  \printfield{eid}}

\DeclareFieldFormat[journal]{journaltitle}{\mkbibemph{#1}}

\DeclareNameAlias{author}{family-given}
\DeclareNameAlias{editor}{family-given}

\DeclareDelimFormat{namedelim}{\addcomma\space}
\DeclareFieldFormat[article,incollection,inproceedings,book]{giveninit}{\mkbibnamegiven{#1}}

\DeclareNameFormat{family-given}{%
  \usebibmacro{name:family-given}
    {\namepartfamily}
    {\namepartgiveni} 
    {\namepartprefix}
    {\namepartsuffix}%
  \usebibmacro{name:andothers}}

\DeclareFieldFormat{pages}{#1}

 \allowdisplaybreaks[4]
\begin{document}

\title{The Dual Majorizing Measure Theorem for Canonical Processes}


\author{Xuanang Hu}
\address{Shandong University,  Jinan, China.}
\email{xuananghu7@gmail.com}

\author{Vladimir V. Ulyanov}
\address{Lomonosov Moscow State University and National Research University Higher School of Economics, Moscow, Russia.}
\email{vulyanov@cs.msu.ru }

\author{Hanchao Wang}
\address{Shandong University, Jinan,  China.}
\email{wanghanchao@sdu.edu.cn}


\date{}


\begin{abstract}
We give a dual, separated-tree formulation of Latała’s majorizing measure theorem for canonical processes with log-concave tails. Under the same assumptions as in Latała’s characterization, we introduce parameterized separation trees and prove that the expected supremum is equivalent, up to constants, to the corresponding tree functional. We also develop a pointwise growth condition, inspired by the contraction principle, which leads to a deterministic polynomial-time algorithm for approximating the expected supremum when the index set is finite.
\end{abstract}

\maketitle

\section{Introduction}

Consider independent symmetric random variables $\left(Y_i\right)_{i \geq 1}$.  Let \(T\) be a subset of $\ell^2 = \bigl\{\, x = (x_i)_{i\ge 1} : \sum_{i=1}^\infty x_i^2 < \infty \,\bigr\}$.
Consider a collection of random variables \(X = (X_t)_{t \in T}\) indexed by \(T\).
If for every \(t \in T\),
\(
X_t = \sum_{i \ge 1} t_i Y_i,
\)
then the process \(X = (X_t)_{t\in T}\) is called a \emph{canonical process} over \(T\).
A central quantity in the study of stochastic processes is the expected supremum $\textsf{E}\sup_{t\in T} X_t$, where
$$\textsf{E}\sup_{t\in T} X_t:=\sup\{\textsf{E}\sup_{t\in F} X_t, F\subset T, F ~\text{finite}\}.$$

Characterizing the boundedness of a stochastic process is often referred to as the majorizing measure theorem. It has a long history, dating back to the seminal result of Fernique \cite[Chapters 4--6]{f75}
. In general, the question of the supremum of process is tricky. However, it was answered in many cases. The most famous result for Gaussian processes is Talagrand's majorizing measure theorem, in which he developed the generic chaining technique to describe the upper bound of stochastic processes as a functional. The interested reader is referred to Chapter~2 of \cite{talbook} for further details.
In our paper, we focus on geometrical methods for providing bounds for the supremum of a process. In 2003,  Gu\'edon and Zvavitch \cite{gz} studied Gaussian processes using labeled-covering trees and packing trees. This idea comes from the work of Talagrand \cite{talacta94}.  This
point of view has been very fruitful in the study of embeddings of subspaces. Ding et al. \cite{ding1} obtained an entirely geometrical characterization for the expected supremum of the Gaussian process, and based on this, they exhibited a strong connection between the cover times of the graphs, Gaussian processes, and Talagrand’s theory of majorizing measures.  As a result, they developed a deterministic polynomial-time algorithm to calculate the cover times of the graphs. Here, the cover time refers  to the maximum of the expected times required to visit each vertex of the graph when random walk begins at an arbitrary vertex. 
The study of supremum and the geometrical characterization of canonical processes plays an essential role in mathematics and applications, 
ranging from high-dimensional statistics and machine learning \cite{bartlett2002rademacher}, 
to empirical process theory and statistical inference \cite{talacta94,talagrand1996new},
as well as signal processing and the theory of random operators \cite{latala2005some,Vershynin2010}.

When $\{Y_i\}_{i\ge 1}$ is a sequence of Gaussian random variables,  Talagrand \cite{talacta87} provides complete control of the size of $ \textsf{E}\sup_{t\in T} X_t$. Talagrand \cite{talacta94} extended his result to the non-Gaussian case. When $\{Y_i\}_{i\ge 1}$ is a sequence of the Bernoulli random variables,  $(X_t)_{t\in T}$ is called the Bernoulli process.   Bednorz and Latała \cite{b14} studied the expected supremum of the Bernoulli process in the spirit of chaining, resolving a longstanding conjecture of Talagrand (``the Bernoulli Conjecture''). Bednorz and Martynek \cite{ref2} obtained a full characterization of the expected supremum of infinitely divisible processes and positively settled the conjecture posed by Talagrand. Recently,   Park and Pham \cite{p24}  obtained that 
 the event that the supremum of a  generalized Bernoulli process  is significantly larger than its expectation can be covered by a set of half-spaces whose sum of measures is small.  

We define the functions
$$
U_i\left(x\right)=-\log \textsf{P}\left\{\left|Y_i\right| \geq x\right\},
$$
 When $U_i\left(x\right)$ are convex for all $i\ge 1$, $ X_t=\sum_{i \geq 1} t_i Y_i$ is called log-concave-tailed canonical process. In our paper, we focus on the study of the expected supremum of log-concave-tailed canonical processes. 

Studying the log-concave-tailed canonical process can go back to the results of Latała \cite{latala97}, where a new version of the Sudakov minoration principle for estimating from below $ \textsf{E}\sup_{t\in T} X_t$ was found. Latała \cite{latala14} formulated and discussed the conjecture that every $d$-dimensional log-concave random vector satisfies the Sudakov minoration principle with a universal constant in \cite{latala14}. Weaker forms of the conjecture (under stronger assumptions) were also proved.  A typical example of a log-concave vector is a
vector uniformly distributed on a convex body, which is of interest in asymptotic geometric analysis. Latała and Tkocz \cite{latala15} introduced the definition of $\alpha$-regular growth of moments  and proved that canonical processes based on random variables with regular moments satisfy the Sudakov minoration principle.  The log-concave-tailed canonical process can be regarded as an example in Latała and Tkocz \cite{latala15}. Further studies can be found in \cite{bogucko} and \cite{latala24}.

The primary purpose of this paper is to provide a dual geometric formulation of Latała’s characterization of the expected supremum for canonical processes with log-concave tails. A standard approach to such problems is Talagrand's generic chaining method \cite{talbook}, in which the so-called growth condition plays a crucial role. This condition may be viewed as a natural extension of the Sudakov minoration principle. Under the same assumptions as those considered here, the suprema of such canonical processes were already completely characterized by Latała \cite{latala97}; this now classical material is also treated in Chapter 8 of \cite{talbook}. This is the content of the majorizing measure theorem in this setting.  Thus, the aim of the present paper is not to obtain a new characterization of the supremum itself, but to provide a dual geometric formulation of this known characterization in terms of separated trees.

However, when one seeks a geometric characterization for canonical processes with log-concave tails, difficulties arise that are more delicate than in the Gaussian case. In the log-concave setting, the growth condition has to be established with respect to a family of distances, and one is therefore led to analyze the underlying geometric structure simultaneously across this entire family. By contrast, for Gaussian processes, it is enough to verify the growth condition with respect to a single prescribed metric.  

The main contributions of this paper are as follows. Firstly,  inspired by van Handel \cite{van1,van2}, we develop a new method to establish the growth condition.
Furthermore, to obtain a geometric characterization of   $\textsf{E}\sup_{t\in T} X_t$, 
we introduce a new tree-based structure, called a \emph{parameterized separation tree}. By a tree, we mean a family of subsets of a set 
$T$ that satisfies certain structural conditions, used to describe how the set is partitioned in the generic chaining method. For definition of a tree, see Section 2. 
This construction allows us to define an associated geometric functional, which provides matching upper and lower bounds for $\textsf{E}\sup_{t\in T} X_t$. The resulting theorem should be viewed as a dual tree representation of Latała’s theorem, rather than as an independent replacement for it.
More precisely, we have the following dual majorizing measure result:

\begin{theorem}\label{dmr}
Under conditions~\eqref{1} and~\eqref{2}, 
\[
\sup_{\mathcal{T}} \inf_{t\in\mathcal{T}}
\sum_{\substack{t \in A \\ A \in \bigcup_{i \ge 0} \mathcal{B}_i}}
\log\!\left(\bigl|c\bigl(p(A)\bigr)\bigr|\right)
\, r^{-\mathbf{j}(A)}
\;\sim_{r}\;
 \, \textsf{E}\sup_{t \in T} X_t ,
\]
where the supremum is taken over all parameterized separation trees $\mathcal{T}$ on $T$, the notation $A \sim_{r} B$ means that there exist positive constants $c(r)$ and $C(r)$, depending only on $r$, such that $c(r)A \le B \le C(r)A$.
\end{theorem}

All relevant definitions and the construction of the parameterized separation tree are presented in Section~2. 

In particular, we introduce the growth condition of points. By combining it with the dual majorizing measure theorem provided in our paper,  we can provide a deterministic polynomial-time algorithm for approximating $\displaystyle \textsf{E}\sup_{t\in T} X_t$, up to universal constants, when $|T|<\infty$. This algorithm is also inspired by Ding et al. \cite{ding1}.

The remainder of this paper is organized as follows. Section 2 first recalls some basic concepts and notation, and lists some lemmas about the concentration inequalities of the log-concave-tailed canonical processes. Section 3 is devoted to the discussion of the growth conditions. In Section 4, we provide the proof of our main results. Section 5 discusses two critical examples, and applications are provided in Section 6.

\section{Preliminaries}

In this section, we first introduce some properties of logarithmically concave, as well as related conclusions from previous studies, which are very helpful in stating our theorems.

Recall that the functions
$$
U_i\left(x\right)=-\log \textsf{P}\left\{\left|Y_i\right| \geq x\right\},
$$
are assumed to be convex throughout the paper. We also assume that $U_i\left(1\right)=1$. Since $U_i\left(0\right)=0$,  we have $U_i^{\prime}\left(1\right)\ge1$ by convexity.  Furthermore, consider the function $\hat{U}_i\left(x\right)$ on $\mathbb{R}$, which is given by
$$
\hat{U}_i\left(x\right)=\left\{\begin{array}{l}
x^2, ~\text { if } 0 \leq|x| \leq 1, \\
2 U_i\left(|x|\right)-1, \text { if }|x| \geq 1.
\end{array}\right.
$$
It is easy to see that the function is convex. For any $u>0$, we define
$$
\mathcal{N}_u\left(t\right)=\sup \left\{\sum_{i \geq 1} t_i a_i: \sum_{i \geq 1} \hat{U}_i\left(a_i\right) \leq u\right\}
$$
and  $$
B\left(u\right)=\left\{t: \mathcal{N}_u\left(t\right) \leq u\right\}.
$$
By Corollary 8.27 in Talagrand \cite{talbook},  we have
\begin{align}\label{| |< B}
    \text { If } u \geq 1 \text { and } t \in B\left(u\right) \text {, we have }(\textsf{E}|X|^u)^{1/u}=\left\|X_t\right\|_u \leq Cu \text {, }
\end{align}
where $C$ denotes a constant. In the sequel, $k,K, c, C$ with and without indices represent constants. We assume that these constants do not necessarily have to be the same every time they appear. If a constant depends on a parameter $\alpha$, we write $c = c(\alpha)$.  We use the notation $A \sim_\alpha B$ to indicate that $c(\alpha)A \le B \le C(\alpha)A$ for some absolute constants $c(\alpha)$ and $C(\alpha)$. We use $A\lesssim_\alpha B$ to indicate that $A\le C(\alpha)B$.

Given $r\geq 4$, we define 
$$
\varphi_j\left(s, t\right)=\inf \left\{u>0 : s-t \in r^{-j} B\left(u\right)\right\},
$$
$$
B_j \left(s,r\right) = \{t\in T;\varphi_j\left(s, t\right) \leq r  \}.
$$
Since $U_i$ is convex, $U_i\left(x\right)$ grows at least as fast as a linear function. We will assume the following regularity conditions, which ensure that $U_i$ does not grow too fast.   For some constant $C_0$, we assume the following throughout the paper:
\begin{align}\label{1}
\forall i \geq 1, \forall s \geq 1, U_i\left(2 s\right) \leq C_0 U_i\left(s\right) 
\end{align}
This condition is often referred to as ``the $\Delta_2$-condition." We will also assume that
\begin{align}\label{2}
\forall i \geq 1, U_i^{\prime}\left(0\right) \geq 1 / C_0,
\end{align}
where $U_i^{\prime}\left(0\right)$ is the right derivative at 0 of the function $U_i\left(x\right)$. 

Talagrand developed the generic chaining method
to capture the relevant geometric structure and the boundedness of stochastic processes.  Below are some notation for the generic chaining method.   For a set $A$,  $|A|$ means the cardinality of $A$. 

\begin{definition}
An admissible sequence of partitions of  $T$ is an increasing sequence $\left(\mathcal{A}_n\right)_{n\geq 0}$ of partitions of $T$ such that $|\mathcal{A}_n| \leq N_n = 2^{2^n}$ for $n\geq 1$ and  $ |\mathcal{A}_0|=1=N_0$. 
\end{definition}

The following theorem represents a part of the majorizing measure theorem for canonical processes with log-concave tails.

\begin{theorem}\label{them 1.1} (See Theorem 8.3.2 in  \cite{talbook})
Under conditions (\ref{1}) and (\ref{2}), for any $T\subset\ell^2$, there exists an admissible sequence $\left(\mathcal{A}_n\right)_{n\geq0}$ of $T$ and $j_n(\cdot) \in\mathbb{Z}$ associated with $\left(\mathcal{A}_n\right)_{n\geq0}$ 
such that
$$
\forall A \in \mathcal{A}_n, \forall s, s^{\prime} \in A, \varphi_{j_n\left(A\right)}\left(s, s^{\prime}\right) \leq 2^n,
$$
$$
c(C_0) \textsf{E} \sup _{t \in T} X_t < r^{-j_0}\leq c(C_0)  r\textsf{E} \sup _{t \in T} X_t ,
$$
$$\forall s, t \in T, \varphi_{j_0}\left(s, t\right) \leq 1,$$
$$
\sup _{t \in T} \sum_{n \geq 0} 2^n r^{-j_n\left(A_n\left(t\right)\right)} \lesssim_{C_0}   r  \textsf{E} \sup _{t \in T} X_t,
$$
for any $r \geq r_{0} $. Here, 
$r_{0} $ is a constant only depending on $C_0$ and $A_n\left(t\right)$ denotes the unique set $A\in \mathcal{A}_n$ such that $t\in A$.

\end{theorem}

In our paper, we provide a geometric characterization of the above theorem, known as the dual majorizing measure theorem. We need to introduce the notion of parameterized separation tree to present our new dual majorizing measure theorem. Firstly,  we recall the definition of the tree. A tree $\mathcal{T}$ in a metric space $(T,d)$ is a  collection of non-empty subsets of $T$ with the following properties:

\begin{itemize}
    \item Given $A, B$ in $\mathcal{T}$, if $A \cap B \neq \varnothing$, then $A \subset B$ or $B \subset A$;
    \item $\mathcal{T}$ has a largest element.
\end{itemize}

If $A, B \in \mathcal{T}$ and $B \subset A, B \neq A$, we say that $B$ is a child of $A$ if
$$
C \in \mathcal{T}, B \subset C \subset A \Rightarrow C=B \text { or } C=A .
$$

Let \( c(B) \) denote the set of children of node \( B \) in the tree. The parent function \( p(A) \) is defined as follows:  
\[
p(A) = 
\begin{cases} 
B & \text{if there exists } B \text{ such that } A \in c(B), \\
A & \text{otherwise (i.e., \( A \) is the root of the tree)}.
\end{cases}
\]

It is convenient to introduce the set 
$$
S_{\mathcal{T}} =\{t \in T:  \{t\} \in \mathcal{T} \text{ or there exist infinite sets } \left(A_i\right) \text{ such that } t\in \cap _{i}A_i\}.
$$

We assume that there exists a subset of $T$ as the largest element in the tree. The children of the largest element are the first generation of the largest element, and the children of each generation form a generation. In this way, the tree structure is composed of each generation of descendants of the largest element. The tree is built through every generation.  For convenience of notation, we denote the largest element as $\mathcal{B}_{-1}$, and subsequent generations are $\{\mathcal{B}_{k}\}_{k\ge 0}$.

For every element $A$ of $\mathcal{T}$, there are two parameters,  $\mathbf{n}:\mathcal{T} \rightarrow \mathbb{N}\bigcup \{-1\}$ and $\mathbf{j}:\mathcal{T} \rightarrow \mathbb{Z}\bigcup \{\infty\}$.  $\mathbf{j}(\textsf{A})=J$, $\mathbf{n}(\textsf{A})=-1$ if $\textsf{A}$ is the largest element of $\mathcal{T}$, and $J$ is a given positive integer such that $\varphi_J(s/4,t/4)\leq 1$ for any $s,t\in T$. (The assignment $\mathbf{n}(A)=-1$ is largely formal, carrying limited mathematical significance.)

\begin{definition}[Parameterized separation tree]
    A parameterized separation tree is a tree $\mathcal{T}$ with two parameters,  $\mathbf{n}:\mathcal{T} \rightarrow \mathbb{N}\bigcup \{-1\}$ and $\mathbf{j}:\mathcal{T} \rightarrow \mathbb{Z}\bigcup \{\infty\}$, satisfying the following conditions:
\begin{enumerate}
    \item $\mathbf{j}(\textsf{A})=J$, $\mathbf{n}(\textsf{A})=-1$ if $\textsf{A}$ is the largest element of $\mathcal{T}$.
    \item $\mathbf{n}\left(A\right)>\mathbf{n}\left(p\left(A\right)\right)$.
    \item $\mathbf{j}\left(A\right)\geq J$ if $p(A)$ is the largest element. 
    \item $B_{\mathbf{j}\left(A\right)}\left(2^{\mathbf{n}\left(A\right)}\right)\subset B_{\mathbf{j}\left(p\left(A\right)\right)}\left(2^{\mathbf{n}\left(p\left(A\right)\right)}\right)$ if $p(A)$ is not the largest element.
    \item $\mathbf{j}\left(A\right)=\infty$ if $|A|<N_{\mathbf{n}\left(A\right)}$; $|c\left(p(A)\right)|=\min \{|A|,N_{\mathbf{n}\left(A\right)}\}$.
    \item If $A_1, A_2 \in c\left(A\right)$, then \\
          $\left(A_1+B_{\mathbf{j}\left(A_1\right)}\left(2^{\mathbf{n}\left(A_1\right)}\right)\right)\bigcap \left(A_2+B_{\mathbf{j}\left(A_2\right)}\left(2^{\mathbf{n}\left(A_2\right)}\right)\right) = \varnothing$.
\end{enumerate}
\end{definition}

\begin{remark}

The parameters of $\mathbf{n}$ and $\mathbf{j}$ are highly dependent on $\left\{B(2^n)\right\}_{n\geq 0}$, and the definition of $\left\{B(2^n)\right\}_{n\geq 0}$ is closely related to the distribution of $Y_i$. Therefore, $\mathbf{n}$ and $\mathbf{j}$ depend on the distribution of $Y_i$. This indicates that the definition of our parameterized separation tree is also dependent on the distribution of $Y_i$.

\end{remark}

\section{Growth condition}

Within the framework of generic chaining, the growth condition plays a crucial role. In this section, we suggest a new growth condition in our setting. In Talagrand's approach, Sudakov's inequality is key to obtaining the growth condition. Inspired by the works of van Handel \cite{van1,van2}, we have developed a new approach to obtaining the growth condition.

Let $\kappa> 8$ be an integer and set $r=2^{\kappa-3}$.  We consider a family of maps $\left(\psi_j\right)_{j \in \mathbb{Z}}$, with the following properties:
$$
\psi_j: T \times T \rightarrow \mathbb{R}^{+} \cup\{\infty\}, \psi_{j+1} \geq r\psi_j \geq 0, \psi_j\left(s, t\right)=\psi_j\left(t, s\right).
$$
Furthermore,  define
$$
B^{\psi}_j\left(t, c\right)=\left\{s \in T:  \psi_j\left(t, s\right) \leq c\right\} ; B^{\psi}_j\left(c\right)= B^{\psi}_j\left(0, c\right).
$$

In this section, we assume that $B^{\psi}_j(2^{n+\kappa})\subset \rho B^{\psi}_{j-1}(2^n)$, $B^{\psi}_j(2^n)=rB^{\psi}_{j+1}(2^n)$ and $B^{\psi}_j(2^{n+1})\subset 2B^{\psi}_j(2^n)$ for $\forall j\in \mathbb{Z}$ and $n\geq 0$. Here, $0<\rho< 1$ can be chosen as an arbitrarily small constant.  
Set 
$$
k^{\psi}_n(A)=\sup\limits_{|S|\leq N_n}\left\{j\in \mathbb{Z}: A\subset \bigcup\limits_{x\in S}B^{\psi}_j(x,2^{n})\right\},
$$
$$
j^{\psi}_n(A)=\sup\left\{j\in \mathbb{Z}: \psi_j(x,y)\leq 2^n \text{ for }~ \forall x,y \in A\right\},
$$
$$
\gamma_{\psi}(T)=\inf\limits_{\mathcal{A}}\sup\limits_{t\in T}\sum\limits_{n\geq 0}2^nr^{-j^{\psi}_n\left(A_n\left(t\right)\right)},
$$
where $\mathcal{A}=(\mathcal{A}_n)_{n\ge 1}$ means any admissible sequence,  $A_n\left(t\right)$ denotes the unique set $A\in \mathcal{A}_n$ such that $t\in A$. From now on, this symbol carries this meaning and in the rest of this paper, it will not be redefined. Obviously, $k^{\psi}_n(A)\geq j^{\psi}_n(A)$ for any $n$ and $A$.

\begin{lemma}\label{gc lemma 1}
If $n>0$, then either $k^{\psi}_n(A)=+\infty$, or for every $\tau\ge1$ there exists a family $(x_i)_{i\le N_n}\subset A$ such that
\[
\psi_{k^{\psi}_n(A)+\tau}(x_i,x_j)>2^n
\quad\text{for all } i\neq j.
\]
\end{lemma}
\begin{proof}
    Assume that $|A|\geq 2$ and $k^{\psi}_n(A)<+\infty$.  For an arbitrary point $x_1\in A$,  there exists $x_2\in A$ such that 
    $$\psi_{k^{\psi}_n(A)+\tau}(x_1,x_2)>2^{n}.$$  
Otherwise, we have $T\subset B^{\psi}_{k^{\psi}_n(A)+\tau}(x_1,2^{n})$. This implies $k^{\psi}_n(A)\geq k^{\psi}_n(A)+\tau$, which is a contradiction.
    
      If $(x_i)_{i\leq m<N_n}$ satisfies $\psi_{k^{\psi}_n(A)+\tau}(x_i,x_j)>2^{n} \text{ for } \forall i\neq j$, then for an arbitrary point $x_{m+1}\in A$, we have  $$\psi_{k^{\psi}_n(A)+\tau}(x_i,x_{m+1})>2^{n} \text{ for } \forall i\leq m. $$ Otherwise $T\subset \bigcup\limits_{i\leq m}B^{\psi}_{k^{\psi}_n(A)+\tau}(x_i,2^{n})$, which implies $k^{\psi}_n(A)\geq k^{\psi}_n(A)+\tau$.  This makes a contradiction again.
    
     Thus, by induction, we can find $(x_i)_{i\leq N_n} \in A$ such that 
$$
\psi_{k^{\psi}_n(A)+\tau}(x_i,x_j)>2^{n} \text{ for } \forall i\neq j.
$$
\end{proof}
Inspired by van Handel \cite{van1,van2}, we develop the following approach to obtain the growth condition, which we refer to as the contraction principle.

\begin{theorem}[Contraction principle]\label{Contraction Principle}
Let $a: 0\le a<1$ be such that
$$
r^{-k^{\psi}_n(A)}\leq C_1\left( ar^{-j^{\psi}_n(A)}+\sup\limits_{x\in A}r^{-s_n(x)} \right)
$$
for every $n\geq 0$ and $A\subset T$. Then
$$
\gamma_\psi(T)\leq Poly(C_1,r) \left( a\gamma_\psi(T) + \sup\limits_{x\in T}\sum\limits_{n\geq 0}2^nr^{-s_n(x)} \right),
$$
where $Poly(x,y)$ is a polynomial of $x$ and $y$ with non-negative coefficients.
\end{theorem}
\begin{proof}
Obviously, we have $k^{\psi}_n(A)\geq j^{\psi}_0(T)$ and $j^{\psi}_n(A)\geq j^{\psi}_0(T)$ for all $n\geq 0$ and $A\subset  T$. Thus, we can assume that $r^{-s_n(x)}\leq r^{-j^{\psi}_0(T)}$ for all $n\geq 0$ and $x\in T$.

Let $(\mathcal{A}_n)$ be an admissible sequence of partitions of $T$. For every $n\geq 1$ and partition element $A_n\in \mathcal{A}_n$, we construct sets $A_n^{ij}$ as follows. 

We first make partition $A_n$ into $n$ segments (here $1\leq i<n$):
$$
A_n^i=\{x\in A_n: 2^{-2i}r^{-j^{\psi}_0(T)}<r^{-s_n(x)}\leq 2^{-2(i-1)}r^{-j^{\psi}_0(T)}\};
$$
$$
A_n^n=\{x\in A_n: r^{-s_n(x)}\leq 2^{-2(n-1)}r^{-j^{\psi}_0(T)}\}.
$$
We can divide  each $A_n^i$ into at most $N_n$ pieces $A_n^{ij}$ such that 
$$
j^{\psi}_{n+3}(A_n^{ij})\geq j^{\psi}_n(A_n^{ij})\geq k^{\psi}_n(A_n^{i})-1,
$$
$$
r^{-k^{\psi}_n(A_n^{ij})}\leq r^{-k^{\psi}_n(A_n^i)}\leq C_1\left(ar^{-j^{\psi}_n(A_n)}+r^{-s_n(x)}+2^{-2(n-1)}r^{-j^{\psi}_0(T)}\right)
$$
for all $x\in A_n^{ij}$. Let $\mathcal{C}_{n+3}$ be the partition generated by all sets $\{A_k^{ij}\}$. Then $$|\mathcal{C}_{n+3}|\leq\prod_{k=1}^{n}k(2^{2^k})^2<N_{n+3}.$$ Set $\mathcal{C}_k=\{T\}$ for $0\leq k \leq 3$, we  have
\begin{align*}
    \gamma_{\psi}(T)&\leq \sup\limits_{x\in T} \sum_{n\geq 0} 2^nr^{-j^{\psi}_n(C_n(t))}\\
    &\leq Poly(C_1,r)\left( a\sup\limits_{t\in T}\sum\limits_{n\geq0}2^nr^{-j^{\psi}_n(A_n(t))} + \sup\limits_{x\in T}\sum_{n\geq 0}2^{n}r^{-s_n(x)} + r^{-j^{\psi}_0(T)}\right),
\end{align*}
where the universal constant only depends on $r$. By definition, we have $T\subset B^{\psi}_{k^{\psi}_0(T)}(x,1)$ for some $x\in l^2$. So we have 
\begin{align*}
r^{-j^{\psi}_0(T)}&\leq r^{-k^{\psi}_0(T)-1}\leq C_1r\left(ar^{-j^{\psi}_0(T)}+\sup\limits_{x\in T}r^{-s_0(x)}\right)
\\
&\leq Poly(C_1,r)\left(a\gamma_\psi(T) + \sup\limits_{x\in T}\sum\limits_{n\geq 0}2^nr^{-s_n(x)}\right).
\end{align*}

Combining the two inequalities above, we obtain the theorem. 

\end{proof}

\begin{remark}
    When applying contraction principle to estimate $\gamma_{\psi}(T)$, we observe that setting $a$ as a sufficiently small constant yields an upper bound for $\gamma_{\psi}(T)$. However, due to the arbitrary nature of constant selection, it is unlikely to achieve a precise bound such as $(1+\epsilon)\gamma_{\psi}(T)$. Furthermore, the proof of this principle relies almost entirely on the structure of the admissible sequence, implying that its applicability extends beyond the specific context presented in our paper.
\end{remark}
Define  $j^{\psi}_{0}:=j^{\psi}_{0}(T)-1$. Now, we define the growth condition. 
  
\begin{definition}[Growth condition]\label{growth conditon}
We say that $\{F_{n,j}\}$ satisfies the growth condition for $C_2$ if the following occurs. Let $\{F_{n, j}\}$ on $T$ be non-decreasing maps from the subsets of $T$ to $\mathbb{R}^{+}$ for $n \geq 0$ and $j\in \mathbb{Z}$. Assume that $F_{n,j}\geq \max\{F_{n+\kappa,j+1},F_{n,j+1}\}$ and $F_{n,j}\leq F_{0,j^{\psi}_0}=F_{n,k}$ for $\forall n \geq 0, j \geq j^{\psi}_{0} \geq k$. Recall that $\kappa> 8$ is an integer and $r=2^{\kappa-3}$. Moreover, we assume that for  any points $(t_i)_{i\leq N_n} \in l^2$, $C_2\geq 1$ and $j>j^{\psi}_0$ satisfying
    $$
    \begin{aligned}
    \psi_{j+1}(t_i,t_{i^\prime})> 2^{n} \text{ for } \forall i\neq i^\prime,   \\
    \exists u, \text{ s.t. } t_i \in C_2B^{\psi}_{j}(u,2^n) \text{ for } \forall i,
    \end{aligned}
    $$
    and for any $i\leq N_n$, given any sets $H_i \subset  B^{\psi}_{j+2}(t_i,2^{n+\kappa})$, 
    $$
    F_{n,j}(\bigcup\limits_{i\leq N_n}H_i)\geq C(r)2^nr^{-j} + \min\limits_{i\leq N_n}F_{n+\kappa,j+2}(H_i)
    $$
   holds.  Here $C(r)$ is a constant such that $C(r)\geq 1/Poly(r)$, and $Poly(r)$ is a polynomial of  $r$ with non-negative coefficients. 
    
\end{definition}

Define
    $$
    K_n(t,x)=\inf_{s\in \mathbb{Z}}\{tr^{-s}+F_{0,j^{\psi}_0}(T)-F_{n,s}\left(B^{\psi}_s\left(x,2^n\right)\right)\}.
    $$
We have
\begin{align*}
K_n(c_1a2^n,x)&\leq c_1a2^nr^{-s_n^a(x)} + F_{0,j^{\psi}_0}(T)-F_{n,s_n^a(x)}\left(B^{\psi}_{s_n^{a}(x)}\left(x,2^n\right)\right)\\
&\leq K_n(c_1a2^n,x) + 2^{-n}F_{0,j^{\psi}_0}(T).
\end{align*}

\begin{lemma}\label{interpolation} We have 
$$
\sup\limits_{x\in T}\sum_{n\geq 0}2^nr^{-s_n^a(x)}\lesssim \frac{1}{r}\left(\frac{ F_{0,j^{\psi}_0}(T)}{a}+r^{-j^{\psi}_0}\right).
$$

\end{lemma}

\begin{proof}
Obviously, $0\leq K_{n}\leq F_{0,j^{\psi}_0}(T)$ and $s_n^{a}(x)\geq j^{\psi}_0$ for all $n\geq 0$.

If $s_{n+\kappa}^a(x)\neq j^{\psi}_0$, we have

   \begin{eqnarray*}
    & &2^{-n}F_{0,j^{\psi}_0}(T)+K_{n+\kappa}(c_1a2^{n+\kappa},x)-K_n(c_1a2^n,x)\\
    &\geq& c_1a(2^\kappa-r)2^nr^{-s_{n+\kappa}^a(x)}+F_{n,s_{n+\kappa}^a(x)-1}(B^{\psi}_{s_{n+\kappa}^a(x)-1}(x,2^{n}))\\& &-F_{n+\kappa,s_{n+\kappa}^a(x)}(B^{\psi}_{s_{n+\kappa}^a(x)}(x,2^{n+\kappa}))\\
    &\geq& c_1a(8r-r)2^nr^{-s_{n+\kappa}^a(x)}+F_{n+\kappa,s_{n+\kappa}(x)}(B^{\psi}_{s_{n+\kappa}^a(x)-1}(x,2^{n}))\\& &-F_{n+\kappa,s_{n+\kappa}^a(x)}(\rho B^{\psi}_{s_{n+\kappa}^a(x)-1}(x,2^{n}))\\
    &\gtrsim& ra2^nr^{-s_{n+\kappa}^a(x)}.
    \end{eqnarray*}

 Otherwise
 \begin{eqnarray*}
     & & 2^{-n} F_{0,j^{\psi}_0}(T)+K_{n+\kappa}(c_1a2^{n+\kappa},x)-K_n(c_1a2^n,x)\\
    &\geq& c_1a(2^\kappa-r)2^nr^{-j^{\psi}_0}+F_{n,j^{\psi}_0-1}(B^{\psi}_{j^{\psi}_0-1}(x,2^{n}))-F_{n+\kappa,j^{\psi}_0}(B^{\psi}_{j^{\psi}_0}(x,2^{n+\kappa}))\\
    &\geq& c_1a(2^\kappa-r)2^nr^{-j^{\psi}_0}+F_{n,j^{\psi}_0-1}(T)-F_{n+\kappa,j^{\psi}_0}(T)\\
    &\gtrsim& ra2^nr^{-s_{n+\kappa}^a(x)}.
 \end{eqnarray*}

The lemma is obtained by summing up these inequalities for all $n\geq 1$.
\end{proof}

\begin{theorem}\label{lower bound}
If the functions $\{F_{n,j}\}$ satisfies the growth condition for $c_2r$, where $c_2>0$ is a sufficiently large constant, then we have
    $$
    \gamma_\psi(T)\lesssim_r\left(F_{0,j^{\psi}_0}(T) + r^{-j^{\psi}_0(T)}\right).
    $$
\end{theorem}

\begin{proof}
    Put
    $$
    \Delta_n(x)=K_{n+\kappa}(c_1a2^{n+\kappa},x)-K_n(c_1a2^n,x),
    $$
    $$
    r^{-s_n(x)}=2^{-n}\left(\Delta_n(x)+2^{-n}F_{0,j^{\psi}_0}(T)+2^{-n}r^{-j^{\psi}_0(T)}\right)+Cr^{-s_{n+\kappa}^a(x)} \text{ for } n\geq 0.
    $$
    
    We want to prove $$r^{-k^{\psi}_n(A)}\leq Poly(r)\left(\left(ac_{2}+\frac{1}{c_{2}}\right)r^{-j^{\psi}_n(A)}+\sup\limits_{x\in A}r^{-s_{n}(x)}\right)$$ for all $n\geq 0$ and $A\subset T$. When $n=0$, it is easy to see because $r^{-k^{\psi}_0(T)}\leq r^{-j^{\psi}_0(T)}$. Thus, we assume that $k^{\psi}_n(A)<+\infty$ and $r^{-j^{\psi}_n(A)}\leq c_{2}r^{-k^{\psi}_n(A)}$. 
    
    Let $n\geq 1$. By Lemma \ref{gc lemma 1} there exist $(x_i)_{i\leq N_n} \in A$ such that 
    $$
    \psi_{k^{\psi}_n(A)+1}(x_i,x_j)> 2^{n} \text{ for } \forall i\neq j.
    $$

    Set $\varpi:= \sup\limits_{x\in A}r^{-s_{n+\kappa}^a(x)}$. If $ \varpi> r^{-(k^{\psi}_n(A)+2)}$,  then the statements in Theorem \ref{lower bound} is derived.

    If $\varpi \leq r^{-(k^{\psi}_n(A)+2)}$, then $s_{n+\kappa}^a(x)\geq k^{\psi}_n(A)+2$ for all $x\in A$. Let
    $$
    \delta = \min\{j^{\psi}_n(A),\min\limits_{x\in A}s_{n+\kappa}^a(x)\}-1=j^{\psi}_n(A)-1
    $$
    then we have 
    $$
    r^{-j^{\psi}_n(A)}+\varpi\geq r^{-(\delta+1)}\geq \frac{r^{-j^{\psi}_n(A)}+\varpi}{2}
    $$
    and
    $$
    r^{-\delta}=r^{-j^{\psi}_n(A)+1}\leq c_{2}rr^{-k^{\psi}_n(A)}.
    $$
    Set $H_i=B^{\psi}_{s_{n+\kappa}^a(x_i)}(x_i,2^{n+\kappa})\subset  B^{\psi}_{k^{\psi}_n(A)+2}(2^{n+\kappa})$. For $\forall y \in H_j$ and $\forall i \leq N_n$, 
    \begin{align*}
        y-x_i&= y-x_j+x_j-x_i\\
        &\in B^{\psi}_{j^{\psi}_n(A)+1}(2^n)+B^{\psi}_{j^{\psi}_n(A)}(2^n)\\
        &\subset B^{\psi}_{j^{\psi}_n(A)-1}(2^n)\subset B^{\psi}_\delta(2^n).
    \end{align*}
    Thus $\bigcup\limits_{i\leq N_n}H_i\subset B^{\psi}_\delta(x_i,2^{n})\subset c_{2}rB^{\psi}_{k^{\psi}_n(A)}(x_i,2^{n})$ for all $i\leq N_n$.

    Then we have
    \begin{align*}
        C(r)2^nr^{-k^{\psi}_n(A)}
        &\leq F_{n,k^{\psi}_n(A)}(\bigcup\limits_{i\leq N_n}H_i)-F_{n+\kappa,k^{\psi}_n(A)+2}(H_j) \text{ for some } j \leq N_n\\
        &\leq F_{n,\delta}(B^{\psi}_\delta(x_j,2^{n}))-F_{n+\kappa,s_{n+\kappa}^a(x_j)}(B^{\psi}_{s_{n+\kappa}^a(x_j)}(x_j,2^{n+\kappa}))\\
        &\leq c_1a2^{n}r^{-\delta}+F_{0,j^{\psi}_0}(T)-K_n(c_1a2^n,x_j)+K_{n+\kappa}(c_1a2^{n+\kappa},x_j)\\
        &-F_{0,j^{\psi}_0}(T)-c_1a2^{n+\kappa}r^{-s_{n+\kappa}^a(x_j)}+2^{-n+1}F_{0,j^{\psi}_0}(T).
    \end{align*}
We have 
    \begin{align*}
    r^{-k_n(A)}&\leq Poly(r)\left(ac_{2}r^{-j^{\psi}_n(A)}+\sup\limits_{x\in A}2^{-n}\left(\Delta_n(x)+2^{-n}F_{0,j^{\psi}_0}(T)\right)+\sup\limits_{x\in A}ar^{-s_n^a(x)}\right)\\
    &\leq Poly(r)\left(ac_{2}r^{-j^{\psi}_n(A)}+\sup\limits_{x\in A}r^{-s_n(x)}\right),
    \end{align*}
    where we can assume $a<1$.

    Combining Theorem \ref{Contraction Principle} and Lemma \ref{interpolation}, we finish the proof of this theorem because we can
choose the appropriate numbers $a$ and $c_{2}$.
\end{proof}

\section{The proof of the main result}

In this section, we present the proof of Theorem \ref{dmr}. The two parameters of parameterized separation trees do not have an explicit quantitative relationship, which makes them difficult to utilize effectively in proofs.  Therefore, we introduce a subclass of parameterized separation trees, the iterative organized trees,  which achieve the desired goal by controlling the two parameters. Through the use of iterative organized trees, we can fully characterize the expected supremum of log-concave-tailed canonical processes.  Before presenting the definition of an iterative organized tree, we need to introduce some other definitions and lemmas. 

Put
\begin{equation}    
\label{def 2.1}
\bar{\varphi}_j\left(s, t\right)=\varphi_j\left(s/4, t/4\right) .
\end{equation}
Let $\Tilde{j}_0$ be the largest number such that for $\forall s, t \in T, \bar{\varphi}_{\tilde{j}_0}\left(s, t\right) \leq 1$.

From equations (8.1) and (8.9) in Chapter~8 of \cite{talbook}
, we have the following lemma. 

\begin{lemma}\label{lemma 1}
$\left(\bar{\varphi}_j\right)_{j\in \mathbb{Z}}$ have the following properties:
$$
\bar{\varphi}_j: T \times T \rightarrow \mathbb{R}^{+} \cup\{\infty\}, \bar{\varphi}_{j+1} \geq \bar{\varphi}_j \geq 0, \bar{\varphi}_j\left(s, t\right)=\bar{\varphi}_j\left(t, s\right),
$$
$$
\forall s, t \in T, \forall j \in \mathbb{Z}, \bar{\varphi}_{j+1}\left(s, t\right) \geq r \bar{\varphi}_j\left(s, t\right) .
$$
\end{lemma}

Moreover, we have

\begin{lemma}\label{lemma 2}
     If  $\bar{\varphi}_{j\left(t\right)}\left(t, t^{\prime}\right) > 2^{n}, 
\bar{\varphi}_{j\left(t^{\prime}\right)}\left(t, t^{\prime}\right) > 2^{n}$, 
  then $$\left(B_{j\left(t\right)}\left(t,2^{n}\right)+B_{j\left(t\right)}\left(2^{n}\right)\right)\cap \left(B_{j\left(t^{\prime}\right)}\left(t^{\prime}, 2^{n}\right)+B_{j\left(t^{\prime}\right)}\left(2^{n}\right)\right)=\varnothing. $$
\end{lemma}

\begin{proof}
     Suppose $j\left(t\right)\leq j\left(t^{\prime}\right)$, then $B_{j\left(t^{\prime}\right)}\left(2 ^ { n  }\right)\subset B _{j\left(t\right)}\left(2^{n}\right)$.
     If $$\left(B_{j\left(t\right)}\left(t,2^{n}\right)+B_{j\left(t\right)}\left(2^{n}\right)\right)\cap \left(B_{j\left(t^{\prime}\right)}\left(t^{\prime}, 2^{n}\right)+B_{j\left(t^{\prime}\right)}\left(2^{n}\right)\right) 
 \neq \varnothing, $$ then there exists a point $x\in \left(B_{j\left(t\right)}\left(t,2^{n}\right)+B_{j\left(t\right)}\left(2^{n}\right)\right)\cap \left(B_{j\left(t^{\prime}\right)}\left(t^{\prime}, 2^{n}\right)+B_{j\left(t^{\prime}\right)}\left(2^{n}\right)\right)$. Thus, we can find $t_1 \in B _{j\left(t\right)}\left(t,2^n\right)$, $t_2 \in B_{j\left(t^{\prime}\right)}\left(t^{\prime},2^{n }\right)$ such that
 $$
 t-t_1\in B _{j\left(t\right)}\left(2^{n}\right), t_1 - t_2 \in B _{j\left(t\right)}\left(2^{n}\right), t_2-t^{\prime}\in B _{j\left(t\right)}\left(2^{n}\right).
 $$
 So $t-t^{\prime}\in 4 B _{j\left(t\right)}\left(2^{n}\right),$ namely $\bar{\varphi}_{j\left(t\right)}\left(t, t^{\prime}\right)\leq 2^{n}$. This contradicts the assumption that $\bar{\varphi}_{j\left(t\right)}\left(t, t^{\prime}\right) > 2^{n}$.
\end{proof}

We assume that $\mathcal{T}$ is a parameterized separation tree.  Recall that there are two parameters in  a parameterized separation tree,  $\mathbf{n}:\mathcal{T} \rightarrow \mathbb{N}\bigcup \{-1\}$ and $\mathbf{j}:\mathcal{T} \rightarrow \mathbb{Z}\bigcup \{\infty\}$.  $\mathbf{j}(\textsf{A})=J$, $\mathbf{n}(\textsf{A})=-1$ if $\textsf{A}$ is the largest element of $\mathcal{T}$. Moreover, if  $|A|<N_{\mathbf{n}\left(c(A)\right)}$,  then $\mathbf{j}\left(c(A)\right)=\infty$,  every set in $c(A)$ contains exactly one point, and does not have descendants.

Now, we define the iterative organized tree $\mathcal{T}$ as $\Gamma _{\left(\mathbf{n}, \mathbf{j},J\right)}\left(T\right)$. Let $\tilde{j}_0$ be the largest number such that $\forall s, t \in T, \bar{\varphi}_{\tilde{j}_0}\left(s, t\right) \leq 1$. If $J\leq \tilde{j}_0$, we put  $\Gamma _{\left(\mathbf{n},  \mathbf{j},J\right)}\left(T\right)=:\Gamma _{\left(\mathbf{n}, \mathbf{j}\right)}\left(T\right)$.

Every generation of the iterative organized tree is denote as $\mathcal{B}_{k}$, and  $\Gamma _{\left(\mathbf{n}, \mathbf{j},J\right)}\left(T\right)=:\{\mathcal{B}_{k}\}_{k\ge 0}$. We define the iterative organized tree through $\mathcal{B}_{k}$.

Firstly, we assume that $\textsf{A}$ is the largest element of $\mathcal{T}$, and set $\mathbf{j}\left(\textsf{A}\right) = +\infty$, $\mathbf{n}\left(\textsf{A}\right) = -1$. We set $\mathcal{B}_{-1}=\{\textsf{A}\}$, and  $\mathcal{B} _{0} $ is the sets of children the set in $\mathcal{B}_{-1}$.  

   Given an integer $n^{(0)}_0> 0$, every generation of the iterative organized tree is constructed as follows. 
   
   If $|\textsf{A}|< N_{n_0^{\left(0\right)}}$, $\mathcal{B} _{0} = \{ \{a\} : a\in \textsf{A}\} = c\left(\textsf{A}\right)$, and for $\{a\}\in \textsf{A}$, $\mathbf{j}(\{a\})=\infty$, and $\mathbf{n}(\{a\})=n^{(0)}_0$.

    If $|\textsf{A}|\ge  N_{n_0^{\left(0\right)}}$,   we  choose arbitrary $N_{n_0^{\left(0\right)}}$ points $\left(t_l\right)_{l\leq N_{n_0^{\left(0\right)}}} \in \textsf{A}$.  We denote  $T _{n_{0}^{\left(0\right)}}\left(\textsf{A}\right) = \{t_l : 0\le  l\leq N_{n_0^{\left(0\right)}}-1\}$. For  any $t\in T _{n_{0}^{\left(0\right)}}\left(A\right)$, there exists a  sequence of integers,  $\left(j\left(t\right)\right)_{t\in T _{n_{0}^{\left(0\right)}}\left(A\right)}$  satisfies that, 
     $$\forall t^{\prime}\neq t\in T _{n_{0}^{\left(0\right)}}\left(A\right),\bar{\varphi}_{j\left(t\right)}\left(t, t^{\prime}\right) > 2^{n_0^{\left(0\right)}}.$$
     
    Define $$\mathbf{j}\left(t\right)=\inf \{j\left(t\right)\in \mathbb{Z}:\bar{\varphi}_{j\left(t\right)}\left(t, t^{\prime}\right) > 2^{n_0^{\left(0\right)}} \text{ for }\forall t^{\prime}\neq t\in T _{n_{0}^{\left(0\right)}}\left(A\right) \text{ and }j\left(t\right)\geq \max\{\Tilde{j}_0,J\}\}.$$

Furthermore, set
\begin{align*}
& A_t = \textsf{A}\cap B _ { j \left( t \right) } \left( t , 2 ^ { n_{0}^{\left(0\right)}  } \right) \text{ for } \forall t\in T_{n_{0}^{\left(0\right)}};\\
&\mathcal{B} _{0} = \{A_{t}:t\in T _{n_{0}^{\left(0\right)}}\} = c\left(\textsf{A}\right);\\
& \mathbf{j}\left(A_{t}\right) = \mathbf{j}\left(t\right),~\mathbf{n}\left(A_t\right)=n_0^{\left(0\right)}.
\end{align*}

According to the inductive method, we define sequentially.  If  $ \mathcal{B} _ {k}$ have been constructed, for $ A_m \in \mathcal{B} _ {k},$  we can choose an arbitrary integer $n_m^{\left(k+1\right)}> \mathbf{n}(A_m)$.

If $|A_m|< N_{n_m^{\left(k+1\right)}}$, we set  $\{ \{a\} : a\in A_m\} = c\left(A_m\right),$ and $ \mathbf{j}_{\{a\}} = \infty~ \mathbf{n}\left(\{a\}\right)=n_m^{\left(k+1\right)}.$

If $|A_m|\ge  N_{n_m^{\left(k+1\right)}}$,   we  choose arbitrary  $N_{n_m^{\left(k+1\right)}}$
points $\left(t_l\right)_{l\leq N_{n_m^{\left(k+1\right)}}}\in A_m$.   We denote  $T _{n_m^{\left(k+1\right)}}\left(A_m\right)  = \{t_l : 0\le  l\leq N_{n_m^{\left(k+1\right)}-1}\}$. For  any $t\in t\in T _{n_{m}^{\left(k+1\right)}}\left(A_m\right)$, there exists a  sequence of integers,  $\left(j\left(t\right)\right)_{t\in T _{n_{m}^{\left(k+1\right)}}\left(A_m\right)}$,  satisfies that, 
     $$\forall t^{\prime}\neq t\in T _{n_m^{\left(k+1\right)}}\left(A_m\right),\bar{\varphi}_{j\left(t\right)}\left(t, t^{\prime}\right) > 2^{n_m^{\left(k+1\right)}},$$
     and 
     $$B_{j\left(t\right)}\left(2^{n_m^{\left(k+1\right)}}\right)\subset B_{\mathbf{j}\left(A_m\right)}\left(2^{\mathbf{n}\left(A_m\right)}\right).$$
    Define 
    \begin{eqnarray*}
\mathbf{j}\left(t\right)&=&\inf \{j\left(t\right)\in \mathbb{Z}: \bar{\varphi}_{j\left(t\right)}\left(t, t^{\prime}\right) > 2^{n_m^{\left(k+1\right)}} 
\text{ and } 
B_{j\left(t\right)}\left(2^{n_m^{\left(k+1\right)}}\right)\subset B_{\mathbf{j}\left(A_m\right)}\left(2^{\mathbf{n}\left(A_m\right)}\right) \\
& &~~~~~\text{ for } \forall t^{\prime} \neq t \in T _{n_{m}^{\left(k+1\right)}}\left(A_m\right)\}.
\end{eqnarray*}

Furthermore, define
\begin{align*}
& A_{t} = A_m\cap B_{j\left(t\right)}\left(t,2^{n_{m}^{\left(k+1\right)}}\right) \text{ for } t\in T _{n_m^{\left(k+1\right)}}\left(A_m\right);\\
& \{A_{t}; t\in T _{n_m^{\left(k+1\right)}}\left(A_m\right)\} = c\left(A_m\right);\\
& \mathbf{j}\left(A_{t}\right) = \mathbf{j}\left(t\right),~\mathbf{n}\left(A_t\right)=n_m^{\left(k+1\right)},
\end{align*}
and  $$\mathcal{B}_{k+1}=\bigcup_{A\in \mathcal{B}_k } c\left(A\right).$$

We can obtain a array $\{n_{m}^{(k)}\}_{k,m\ge 0}\subset \mathbb{N}$, which  is related to $\{\mathcal{B}_{k}\}_{k\ge 0}$. For each $k\ge 0$, if $A\in \mathcal{B}_k$, $\mathbf{n}(A)\in \{n^{(k)}_{m}\}_{m\ge 0}$.

\begin{figure}[htbp]
    \centering
    \includegraphics[width=0.6\textwidth]{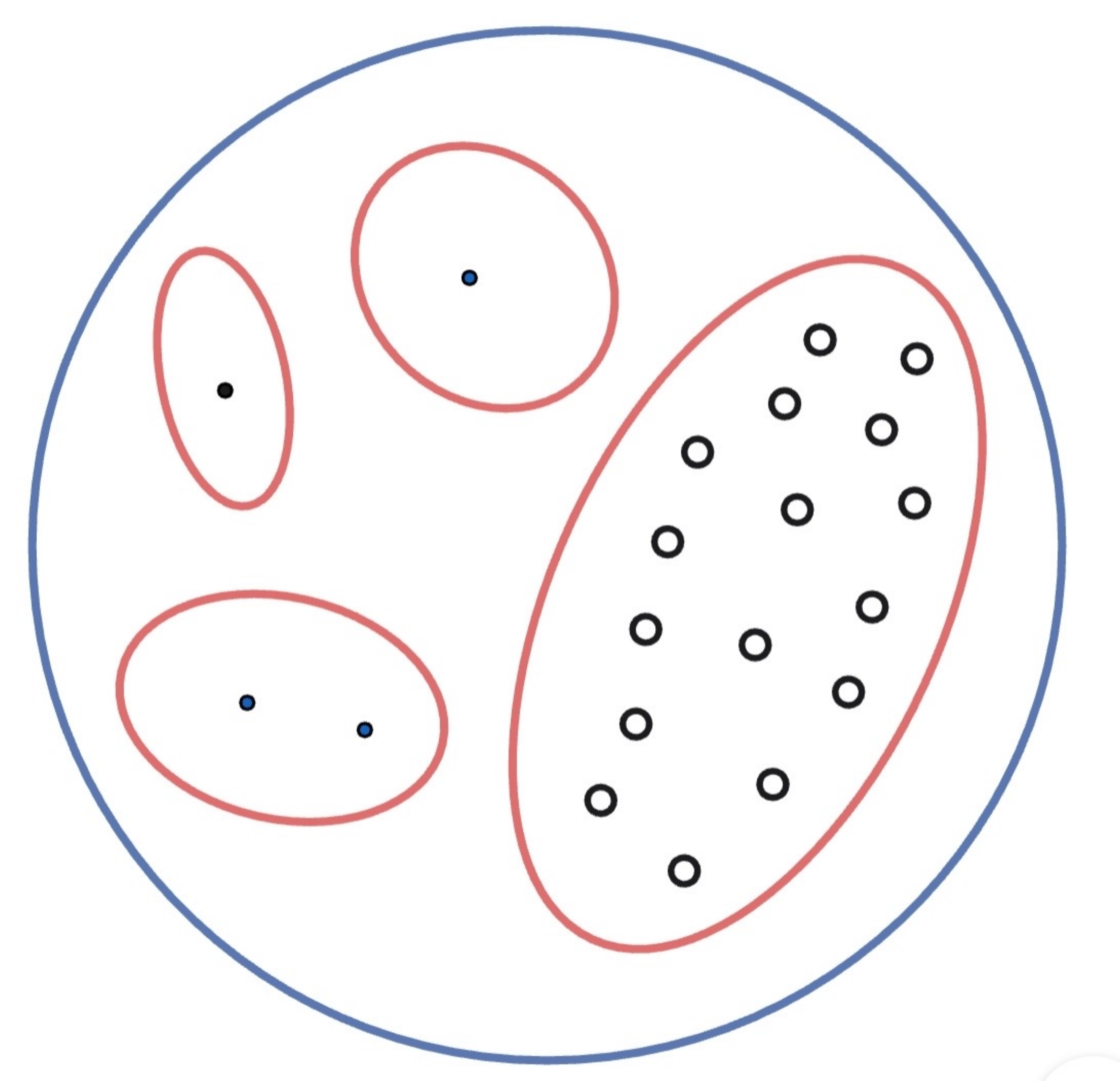}
    \caption{ The iterative organized tree.}
    \label{fig:demo}
\end{figure}

Define  the size of the tree:
$$
\rho\left(\Gamma _{\left(\mathbf{n}, \mathbf{j},J\right)}\left(T\right)\right)=\inf _{t \in \Gamma _{\left(\mathbf{n}, \mathbf{j},J\right)}\left(T\right)}\sum_{t\in A, A\in \cup_{i\ge 0}\mathcal{B}_i}\log\left(\left|c\left(p\left(A\right)\right)\right|\right)r^{-\mathbf{j}\left(A\right)}.
$$
We will obtain the following result. 

\begin{theorem}\label{mr}
Under conditions (\ref{1}) and (\ref{2}) we define the size of $T$:
$$
Size\left(T\right)=\sup \{\rho\left(\Gamma _{\left(\mathbf{n}, \mathbf{j}\right)}\left(T\right)\right): \Gamma _{\left(\mathbf{n}, \mathbf{j}\right)}\left(T\right)\text{ is a iterative organized tree of }~T\}.
$$
Then we have
$$
     Size\left(T\right) \sim_r   \textsf{E}  \sup _{t \in T} X_t.
$$ 
\end{theorem}

To prove Theorem \ref{mr}, we require the following functional $I_\mu\left(t\right)$, often referred to as Fernique's functional. Define $\bar{j}_0$ as the largest number such that
$$
\forall s, t \in T, \varphi_j\left(s, t\right) \leq 1.
$$

Given $t \in T$ and a probability measure $\mu$ on $T$,   define $\bar{j}_0\left(t\right)=\bar{j}_0 \text{ for }\forall t\in T$, and for $n \geq 1$,

\begin{equation}
\label{new11}
\bar{j}_n\left(t\right)=\sup \left\{j \in \mathbb{Z} ; \mu\left(\left\{s \in T ; \varphi_j\left(t, s\right) \leq 2^n\right\}\right) \geq N_n^{-1}\right\},
\end{equation}
so that the sequence $\left(\bar{j}_n\left(t\right)\right)_{n \geq 0}$ increases. Then we set
$$
I_\mu\left(t\right)=\sum_{n \geq 0} 2^n r^{-\bar{j}_n\left(t\right)} .
$$

In fact, we have to argue that the following inequalities hold:
\begin{equation}\label{stp}
\textsf{E} \sup _{t \in T} X_t \mathop{\lesssim_r}\limits^{\left(3\right)} Size\left(T\right) \mathop{\lesssim_r}\limits^{\left(2\right)} \sup _{\mu} \int_T I_\mu\left(t\right) \mathrm{d} \mu\left(t\right) \mathop{\lesssim_r}\limits^{\left(1\right)}  \textsf{E} \sup _{t \in T} X_t.
\end{equation}
The proof of inequality (1) is almost identical to the proof in Section 10.15 of \cite{talbook}. Inequalities (2) and (3) are the contributions of our paper. By the above inequalities, we have the following result.

\begin{theorem} \label{them 1.2}
Under conditions (\ref{1}) and (\ref{2}), we have
$$
    \sup _{\mu} \int_T I_\mu\left(t\right) \mathrm{d} \mu\left(t\right) \sim_r   \textsf{E}  \sup _{t \in T} X_t. 
$$
\end{theorem}

\subsection{The proof of inequality (1)}
Inequality (1) follows by the following theorem.

\begin{theorem}\label{them 2.1}
Under conditions (\ref{1}) and (\ref{2}), given a probability measure $\mu$ on $T$, we have
$$
\int_T I_\mu\left(t\right) \mathrm{d} \mu\left(t\right) \lesssim_r \textsf{E} \sup _{t \in T} X_t.
$$
\end{theorem}
\begin{proof}
Let $\displaystyle S=\textsf{E}\sup _{t\in T}X_t$. According to Theorem \ref{them 1.1}, we know that there exists a number $j_0$ such that
$$
c \textsf{E} \sup _{t \in T} X_t < r^{-j_0}\leq c r\textsf{E} \sup _{t \in T} X_t ,
$$
$$
\forall s, t \in T, \varphi_{j_0}\left(s, t\right) \leq 1.
$$
Thus,  the definition of $\bar{j}_0$ and $\bar{j}_0\left(t\right)$ in (\ref{new11}) yields $$r^{-\bar{j}_0\left(t\right)} \leq CrS$$ for $\forall t \in T$.

Consider the admissible sequence $\left(\mathcal{A}_n\right)$ of $T$ and $j_n\left(A\right)$ in Theorem \ref{them 1.1},  for $t \in A$, 
$$
A \subset\left\{s \in T: \varphi_{j_n\left(A\right)}\left(s, t\right) \leq 2^n\right\} \subset\left\{s \in T: \varphi_{j_n\left(A\right)}\left(s, t\right) \leq 2^{n+1}\right\}.
$$

Thus, by the definition of $\bar{j}_{n+1}\left(t\right)$ in (\ref{new11}),  if $\mu\left(A\right) \geq N_{n+1}^{-1}$ then $\bar{j}_{n+1}\left(t\right) \geq j_n\left(A\right)$. And we have
$$
\int_A 2^{n+1} r^{-\bar{j}_{n+1}\left(t\right)} \mathrm{d} \mu\left(t\right) \leq 2 \int_A 2^n r^{-j_n\left(A_n\left(t\right)\right)} \mathrm{d} \mu\left(t\right) .
$$

Moreover, by $\bar{j}_{n+1}\left(t\right) \geq \bar{j}_0\left(t\right)=\bar{j}_0$, if $\mu\left(A\right)<N_{n+1}^{-1}$, then we have
$$
\int_A 2^{n+1} r^{-\bar{j}_{n+1}\left(t\right)} \mathrm{d} \mu\left(t\right) \leq 2^{n+1} r^{-\bar{j}_0} N_{n+1}^{-1} .
$$

Summation over all $A \in \mathcal{A}_n$ implies (using in the second term that card $\mathcal{A}_n \leq N_n$ and $N_n N_{n+1}^{-1} \leq N_n^{-1}$)
$$
\int_T 2^{n+1} r^{-\bar{j}_{n+1}\left(t\right)} \mathrm{d} \mu\left(t\right) \leq 2 \int_T 2^n r^{-j_n\left(A_n\left(t\right)\right)} \mathrm{d} \mu\left(t\right)+2^{n+1} r^{-\bar{j}_0} N_n^{-1} .
$$

The summation of $n \geq 0$ implies that
$$
\int_T \sum_{n \geq 1} 2^n r^{-\bar{j}_n\left(t\right)} \mathrm{d} \mu\left(t\right) \leq C S+C r^{-\bar{j}_0} .
$$

Since $2^n r^{-\bar{j}_n\left(t\right)} \leq C r^{-\bar{j}_0} \leq C r S$ for $n=0$ or $n=1$, we have 
  $$\int_T I_\mu\left(t\right) \mathrm{d} \mu\left(t\right) \leq C r \textsf{E} \sup _{t \in T} X_t.$$
\end{proof}

By Theorem \ref{them 2.1}, we obtain $$\sup _{\mu} \int_T I_\mu\left(t\right) \mathrm{d} \mu\left(t\right) \leq Cr\textsf{E} \sup _{t \in T} X_t.$$

\subsection{The proof of inequality (2)}

The key points of the proof are given in the following proposition.

\begin{proposition}\label{proposition 1}
For $\forall n \geq 0\text{ and }\forall A\neq A^\prime \in \mathcal{B}_{n}$, we have $$\left(A+B _ {\mathbf{j} \left(A\right) } \left(  2 ^ { \mathbf{n}  \left(A\right)  } \right)\right)\cap \left(A^{\prime}+B _ { \mathbf{j}  \left(A^{\prime}\right) } \left(  2 ^ { \mathbf{n}  \left(A^{\prime}\right) } \right)\right)=\varnothing.$$

\end{proposition}

\begin{proof}
    We may assume $|\mathcal{B}_0|>0$.  Then we know that the proposition holds by Lemma \ref{lemma 2} when $n=0$.

 Assuming that this proposition holds for $n$, it is easy to know that for $\forall A\neq A^\prime \in \mathcal{B}_{n+1}$. If $p\left(A\right)=p\left(A^\prime\right)$, we have 
    $$\left(A+B _ { \mathbf{j}  \left(A\right) } \left(  2 ^ { \mathbf{n}  \left(A\right)  } \right)\right)\cap \left(A^{\prime}+B _ { \mathbf{j}  \left(A^{\prime}\right) } \left(  2 ^ { \mathbf{n} \left(A^{\prime}\right) } \right)\right)=\varnothing$$ 
    by Lemma \ref{lemma 2}. If $p\left(A\right)\neq p\left(A^\prime\right)$, we have
     $$\left(p\left(A\right)+B _ { \mathbf{j}  \left(p\left(A\right)\right) } \left(  2 ^ { \mathbf{n}  \left(p\left(A\right)\right)  } \right)\right)\cap \left(p\left(A^{\prime}\right)+B _ { \mathbf{j}  \left(p\left(A^{\prime}\right)\right) } \left(  2 ^ { \mathbf{n}  \left(p\left(A^{\prime}\right)\right)  } \right)\right)=\varnothing$$ by assumption.
     By means of $B_{\mathbf{j}(A)}\left(2^{\mathbf{n} \left(A\right)}\right)\subset B _ { \mathbf{j}(p\left(A\right))}\left(2 ^ { \mathbf{n} \left(p\left(A\right)\right)}\right)$, we have $$\left(A+B_{\mathbf{j} \left(A\right)}\left(2^{\mathbf{n} \left(A\right)}\right)\right)\subset \left(p\left(A\right)+B_{\mathbf{j} \left(p\left(A\right)\right)}\left(2^{\mathbf{n} \left(p\left(A\right)\right)}\right)\right).$$
     Thus $$\left(A+B _ { \mathbf{j} \left(A\right) } \left(  2 ^ { \mathbf{n}  \left(A\right)  } \right)\right)\cap \left(A^{\prime}+B _ { \mathbf{j}  \left(A^{\prime}\right) } \left(  2 ^ { \mathbf{n}  \left(A^{\prime}\right) } \right)\right)=\varnothing.$$
     Therefore, the proposition holds for $n+1$. By induction, we complete the proof. 
\end{proof}

\begin{theorem}\label{them 2.2}
    Given an arbitrary  iterative organized tree $\Gamma _{\left(\mathbf{n}, \mathbf{j}\right)}\left(T\right)$, we have $$\rho\left(\Gamma _{\left(\mathbf{n}, \mathbf{j}\right)}\left(T\right)\right)\lesssim_r\sup _{\mu} \int_T I_\mu\left(t\right) \mathrm{d} \mu\left(t\right).$$\end{theorem}

\begin{proof}
   
   $\Tilde{j}_0\geq \bar{j}_0 $ is obvious from the definitions. Therefore, what we need to prove is the case that $\mathbf{n}(A)>0$ when $A\in \mathcal{B}_n$.
   
   For $t\in S_{\Gamma _{\left(\mathbf{n}, \mathbf{j}\right)}\left(T\right)}$, let $A
_n\left(t\right)$ be the set in $\mathcal{B}_n$ that contains the point $t$. If $|c\left(p\left(A_{n}\left(t\right)\right)\right)|<N _{\mathbf{n}\left(A_n\right)}$, then $$\log \left(|c\left(p\left(A_n\left(t\right)\right)\right)|\right)r^{-\mathbf{j}\left(A_{n}\left(t\right)\right)}=0.$$
Otherwise, we have $\mu \left(A_{n}\left(t\right)\right)\leq N_{\mathbf{n}\left(A_{n}\left(t\right)\right)}^{-1}$.

   Consider the case where $|c\left(p\left(A_n\left(t\right)\right)\right)|\geq N _{\mathbf{n}\left(A_n\right)}$. By Proposition \ref{proposition 1}, for $\forall n \geq 0\text{ and }\forall A\neq A^\prime \in \mathcal{B}_{n}$, we have $$\left(A+B _ { \mathbf{j} \left(A\right) } \left(  2 ^ { \mathbf{n} \left(A\right)  } \right)\right)\cap \left(A^{\prime}+B _ { \mathbf{j} \left(A^{\prime}\right) } \left(  2 ^ { \mathbf{n} \left(A^{\prime}\right) } \right)\right)=\varnothing.$$
    Thus we have $$\left(t+B _ { \mathbf{j} \left(A_{n}\left(t\right)\right)}\right) \cap A^\prime = \varnothing\text{ for } \forall A^\prime \neq A_n\left(t\right) \in \mathcal{B}_n, $$
    which implies that $$\mu \left(t+B _ { \mathbf{j} \left(A_{n}\left(t\right)\right) } \left(  2 ^ { \mathbf{n}\left(A_{n}\left(t\right)\right) } \right)\right)\leq \mu\left(A_{n}\left(t\right)\right)\leq N_{\mathbf{n}\left(A_{n}\left(t\right)\right)}^{-1}.$$ 
    By definition of $\bar{j}_n\left(t\right)$, we have
    $$2^{\mathbf{n}\left(A_{n}\left(t\right)\right)}r^{-\mathbf{j}\left(A_n\left(t\right)\right)}\leq 2^{\mathbf{n}\left(A_{n}\left(t\right)\right)}r^{-\Bar{j}_{\mathbf{n}\left(A_{n}\left(t\right)\right)} \left(t\right)}.$$
     Therefore, by the construction of tree, we have $$\log \left(|c\left(p\left(A_n\left(t\right)\right)\right)|\right)r^{\mathbf{j}\left(A_n\left(t\right)\right)}\leq 2^{\mathbf{n}\left(A_{n}\left(t\right)\right)}r^{-\Bar{j}_{\mathbf{n}\left(A_{n}\left(t\right)\right)} \left(t\right)}.$$
    Combining the above cases, we obtain $$\rho\left(\Gamma _{\left(\mathbf{n}, \mathbf{j}\right)}\left(T\right)\right) \leq I_{\mu}\left(t\right) ~\text{for}~\forall t\in S_{\Gamma _{\left(\mathbf{n}, \mathbf{j}\right)}\left(T\right)}\left(T\right).$$ 
By definition of $\mu$, we have $$\int_{T}I_{\mu}\left(t\right)\mu\left(dt\right)=\int_{S_{\Gamma _{\left(\mathbf{n}, \mathbf{j}\right)}\left(T\right)}\left(T\right)}I_{\mu}\left(t\right)\mu\left(dt\right).$$
Thus we have
   $$\rho\left(\Gamma _{\left(\mathbf{n}, \mathbf{j}\right)}\left(T\right)\right) \leq  \int_{T}I_{\mu}\left(t\right)\mu\left(dt\right),$$ which implies $$Size\left(T\right) \leq \sup _{\mu} \int_T I_\mu\left(t\right).$$    
\end{proof}

\begin{remark}
The construction of the iterative organized tree is key to proving the theorem \ref{them 2.2}.
\end{remark}

\subsection{The proof of inequality (3)}

We need to introduce the following functionals to complete the proof. In the definition of an iterative organized tree,  $\mathbf{j}(\textsf{A})=\tilde{j}_0$, where $\textsf{A}$ is the largest element of the tree. We assume $J=j+1+\frac{1}{\kappa}(n_{0}^{\left(0\right)}-n)$. When $j\ge \tilde{j}_0$, we set
\begin{equation}
\label{nf}
F_{n,j}\left(T\right)=\sup \left\{\rho\left( \Gamma _{\left(\mathbf{n}, \mathbf{j},J\right)}\left(T\right)\right):n_{0}^{\left(0\right)}\geq n,  \Gamma _{\left(\mathbf{n}, \mathbf{j},J\right)}\left(T\right)\text{ is an iterative organized tree of }~T  \right\}.
\end{equation}
When $j<\tilde{j}_0$,  we set $$F_{n,j}\left(T\right)=\sup \left\{\rho\left( \Gamma _{\left(\mathbf{n}, \mathbf{j}\right)}\left(T\right)\right):n_{0}^{\left(0\right)}\geq 0, \Gamma _{\left(\mathbf{n}, \mathbf{j}\right)}\left(T\right)\text{ is an iterative organized tree of }~T  \right\}. $$
We need the following results to prove our theorem.

\begin{theorem}[See Theorem 2.7.14 in Talagrand \cite{talbook}]\label{them 2.4}
    Consider a process $\left(X_t\right)_{t \in T}$, which is assumed to be centered but does not need to be symmetric. For $n \geq 1$, consider the distance $\delta_n$ on $T$ given by $\delta_n\left(s, t\right)=\left\|X_s-X_t\right\|_{2^n}$. Denote by $\Delta_n\left(A\right)$, the diameter of a subset $A$ of $T$ for the distance $\delta_n$.
Consider an admissible sequence $\left(\mathcal{A}_n\right)_{n \geq 0}$ of partitions of $T$. Then,
$$
\textsf{E} \sup _{s, t \in T}\left|X_s-X_t\right| \lesssim\sup _{t \in T} \sum_{n \geq 0} \Delta_n\left(A_n\left(t\right)\right),
$$
moreover, given $u>0$ and the largest integer $k$ with $2^k \leq u^2$, we have
$$
\textsf{P}\left(\sup _{s, t \in T}\left|X_s-X_t\right| \geq c\Delta_k\left(T\right)+\sup _{t \in T} \sum_{n \geq 0} \Delta_n\left(A_n\left(t\right)\right)\right) \leq C \exp \left(-u^2\right).
$$
\end{theorem}

\begin{lemma}[See Lemma 8.3.9 in Talagrand \cite{talbook}]\label{lemma 2.3}
Under (\ref{1}), given $\rho>0$, there exists $r_0$ depending on $C_0$ in (\ref{2}) and $\rho$ only, such that if $r \geq r_0$, for $u \in \mathbb{R}^{+}$we have
\begin{equation}\label{eq:lemma 2.3}
   B\left(8 r u\right) \subset \rho r B\left(u\right) . 
\end{equation}
\end{lemma}

Using Lemma \ref{lemma 2.3} for $u=2^n\left(\right.$ and since $\left.2^\kappa=8 r\right)$, we obtain that $B\left(2^{\kappa+n}\right) \subset \rho r B\left(2^n\right)$. So we can assume that the $\left(\varphi_j\right)_{j\in \mathbb{Z}}$ in Section 2 are $(\bar{\varphi_j})_{j\in \mathbb{Z}}$ defined in Definition \ref{def 2.1}.

\begin{proposition}\label{prop 2}
   Recall Definition \ref{growth conditon}. If the functions $F_{n, j}$ satisfy the growth condition, then 
   $$
   \textsf{E} \sup _{t \in T} X_t \lesssim_r \left(r Size\left(T\right)+r^{-\tilde{j}_0}\right).
   $$
\end{proposition}

\begin{proof}
    For $s, t \in A \in \mathcal{A}_n$ we have $s-t \in 8r^{-j_n\left(A\right)} B\left(2^n\right)$. So by (\ref{| |< B}), we have $$\left\|X_s-X_t\right\|_{2^n}=\left\|X_{s-t}\right\|_{2^n} \leq C 2^n r^{-j_n\left(A\right)}.$$
    This means that the diameter of $A_n\left(t\right)$ for the distance of $L^p$ with $p=2^n$ is  smaller than $ C 2^n r^{-j_n\left(A\right)}$. In fact,  $$\textsf{E} \sup _{t \in T/4} X_t=\frac{1}{4}\textsf{E} \sup _{t \in T} X_t.$$
    Combining Theorem \ref{them 2.4} with Theorem \ref{lower bound}, we obtain this proposition. 
\end{proof}

Furthermore, we have

\begin{proposition}\label{prop 3}
    Recall Definition \ref{growth conditon}. Let $\{F_{n,j}\}$ be what we define in (\ref{nf}), then $\{F_{n, j}\}$ satisfies the growth condition for $(\bar{\varphi}_j)_{j\in \mathbb{Z}}$ under conditions (\ref{1}) and (\ref{2}). 
\end{proposition}

\begin{proof}
Obviously, functions $F_{n,j}$ is a non-decreasing map from the subsets of $T$ to $\mathbb{R}^{+}$. We also have $$F_{n,j}\geq \max\{F_{n+\kappa,j+1},F_{n,j+1}\}$$ and $$F_{n,j}\leq F_{n,k}=F_{0,j_0}$$ for $\forall n \geq 0$ and $\forall j,k\in \mathbb{Z}$ such that $j \geq j_{0}=\Tilde{j}_{0}(T)-1\geq k$.

Using Lemma~\ref{lemma 2.3} with $u=2^n$ and the identity $2^\kappa = 8r$,
 we obtain $B\left(2^{\kappa+n}\right) \subset \rho r B\left(2^n\right)$. Therefore, $H_{\ell} \subset t_{\ell}+4\rho r^{-j-1} B\left(2^n\right)$. We can assume that $\rho \leq \frac{1}{4}$ so that $H_{\ell} \subset t_{\ell}+ r^{-j-1} B\left(2^n\right)$. We  set 
  $$\bigcup_{l\leq N_{n}} H_l=\mathcal{B}_{-1},\{H_{l}\}_{l\leq N_{n}} = \mathcal{B}_{0},~\mathbf{n}(H_l)=n \text{ for }\forall l\leq N_n.$$
Then it is easy to see that $\textbf{j}\left(H_{l}\right) \leq j+1$. For $\forall \epsilon >0$, there exists $\Gamma _{\left(\mathbf{n}, \mathbf{j}\right)}\left(H_{l}\right)$ with $\mathbf{n}\left(A\right)\geq n+\kappa, \mathbf{j}\left(A\right) \geq J = j+3+(1/\kappa)(n(A)-n-\kappa) \text{ for } \forall A \in \mathcal{B}_{0}$ such that 
 $$\rho\left(\Gamma _{\left(\mathbf{n}, \mathbf{j}\right)}\left(H_{l}\right)\right)\geq F_{n+\kappa,j+2}\left(H_{l}\right)-\epsilon$$ 
  for $\forall l\leq N_n$. We can define an iterative organized tree by induction.  We separate every set of $\mathcal{B}_0$, and obtain $\mathcal{B}_1$ following the previously presented procedure,  We need to demonstrate $B_{\mathbf{j}\left(A\right)}\left(2^{\mathbf{n}\left(A\right)}\right)\subset B_{j+1}\left(2^{n}\right) \text{ for } \forall A\in \mathcal{B}_{1}$. There exists a number $s\geq 1$ such that $n+\kappa s\leq \mathbf{n}(A)< n+\kappa(s+1)$. So we have
 $$
B_{\mathbf{j}\left(A\right)}\left(2^{\mathbf{n}\left(A\right)}\right)\subset B_{j+2+s}\left(2^{n+\kappa(s+1)}\right)\subset (\rho r)^{s+1} B_{j+2+s}\left(2^{n}\right) \subset B_{j+1}\left(2^{n}\right).
 $$ 
 And we have $B_{\mathbf{j}\left(A\right)}\left(2^{\mathbf{n}\left(A\right)}\right)\subset B_{j+1}\left(2^{n}\right) \text{ for } \forall A\in \mathcal{B}_{1}$.

 If we have already defined $\mathcal{B}_n$, we can obtain  $\mathcal{B}_{n+1}$ through $\mathcal{B}_n$. We set $\bigcup _{l}\mathcal{B}_{n}\left(H_l\right)=\mathcal{B}_{n+1}$. Then  $\mathcal{B}_{-1}\bigcup\mathcal{B}_{0}\bigcup \left(\bigcup _{n\geq 1}\mathcal{B}_{n}\right)$ is a new tree $\Gamma _{\left(\mathbf{n}, \mathbf{j}\right)}\left(\bigcup_{l\leq N_{n}} H_{l}\right)$. And we have 
 $$\rho \left(\Gamma _{\left(\mathbf{n}, \mathbf{j}\right)}\left(\bigcup_{\ell \leq N_n} H_{l}\right)\right) \geq \frac{1}{r}2^n r^{-j} + \min _{\ell \leq N_n} \rho\left(\Gamma _{\left(\mathbf{n}, \mathbf{j}\right)}\left(H_{l}\right)\right).
 $$
 Thus we have
 \begin{align*}
 F_{n, j}\left(\bigcup_{\ell \leq N_n} H_{\ell}\right) &\geq \rho \left(\Gamma _{\left(\mathbf{n}, \mathbf{j}\right)}\left(\bigcup_{\ell \leq N_n} H_{l}\right)\right) \\
 &\geq 2^n r^{-j-1}+\min _{\ell \leq N_n} F_{n+\kappa,j+2}\left(H_{\ell}\right) - \epsilon, \text{ for }\forall \epsilon > 0,
 \end{align*}
 which implies the growth condition.

\end{proof}

\begin{theorem}\label{them 2.5}
    Under conditions (\ref{1}) and (\ref{2}), $$\textsf{E} \sup _{t \in T} X_t \lesssim_r Size\left(T\right).$$
\end{theorem}

\begin{proof}
    Combining Proposition \ref{prop 3} and Proposition \ref{prop 2}, it is clear that the only thing we need to prove is $r^{\Tilde{j}_0}\leq CSize\left(T\right)$. To obtain this inequality,  we can choose $n_0^{(0)}=0$ and then $\mathbf{j}(t)\leq \Tilde{j}_0$ for an arbitrary point $t\in T$ .
\end{proof}

For a parameterized separation tree, we introduce

\begin{equation}
\label{3.24-1}
\tau \left(\mathcal{T}\right)=\inf _{t \in \mathcal{T}}\sum_{t\in A, A\in \cup_{i\ge 0}\mathcal{B}_i}\log\left(\left|c\left(p\left(A\right)\right)\right|\right)r^{-\mathbf{j}\left(A\right)}, 
\end{equation}

and we have 

\begin{theorem}\label{them 3.1}
    Under (\ref{1}) and (\ref{2}), we have
    $$
     \sup \{\tau \left(\mathcal{T}\right); \mathcal{T} \text{ parameterized separation tree of } T\}\sim_r\sup _{\mu} \int_T I_\mu\left(t\right) \mathrm{d} \mu\left(t\right).
    $$
\end{theorem}
\begin{proof}
    In fact, it is obvious that any iterative organized tree $\Gamma_{\left(\mathbf{n}, \mathbf{j},J\right)}(T)$ is a parameterized separation tree with two parameters, $\mathbf{n}$ and $\mathbf{j}$, which are defined the same as in the definition of $\Gamma_{\left(\mathbf{n}, \mathbf{j},J\right)}(T)$, and we can denote this parameterized separation tree by $\mathcal{T}_{n_0^{\left(0\right)}}$. Therefore, it is easy to see that  $C\rho\left(\Gamma_{\left(\mathbf{n}, \mathbf{j},J\right)}(T)\right) = \tau \left(\mathcal{T}_{n_0^{\left(0\right)}}\right)$ for some constant $C$. So the first inequality is evidently true.

    The proof of the second inequality is the same as the proof of Theorem \ref{them 2.2}.
\end{proof}
    
    Combined with Theorem \ref{them 2.1}, Theorem \ref{them 2.5}, and Theorem \ref{them 3.1}, we prove Theorem \ref{dmr}.

\section{ Important example}

  An iterative organized tree is essential in our proof. Its construction can be regarded as a complex algorithm. Therefore, it is necessary for us to discuss the construction of such trees in the context of specific examples. 
 When $U_i\left(x\right)=x^p$ for all $x\ge 0$ and $i$, 
                $$\textsf{P}(|Y_i|\ge x)=\exp\{-x^p\}$$
   where $1\le p\le 2$, it is an important special case for log-concave tailed distributions. In this section, we discuss two special cases: $p=1$ and $p=2$. We denote by $B_p$ the unit ball in $l^p$ for $p=1$ or $p=2$. 

In this section, the structure of iterative organized trees can be described in more details. Therefore, the resulting trees differ slightly from those in the
previous section, and some results need to be proven again.

\subsection{Case 1: $p=2$.}

Then it is easy to see $x^2\leq \hat{U}_i\left(x\right)\leq 2x^2$ and 
$$
\sqrt{u/2}\left\| t \right\|_2 \leq \mathcal{N}_u\left(t\right)\leq \sqrt{t}\left\| t \right\|_2 ,  
$$
which implies
\begin{equation}\label{case_p_2}
\sqrt{u}B_2\subset B\left(u\right)\subset \sqrt{2u}B_2.    
\end{equation}

In this case, we define a tree $\mathcal{T}^{(2)}$:(In this example, $r\geq 8.$)
\begin{enumerate}
    \item $\mathbf{j}(A)=J$, $\mathbf{n}(A)=-1$ if $A$ is the largest element of $\mathcal{T}^{(2)}$.
    \item $\mathbf{n}\left(A\right)>\mathbf{n}\left(p\left(A\right)\right)$.
    \item $\mathbf{j}\left(A\right)\geq \Tilde{j}_0$ if $p(A)$ is the largest element. 
    \item $r^{\mathbf{j}(A)-\mathbf{j}(p(A))}\geq 2^{\frac{\mathbf{n}(A)-\mathbf{n}(p(A))+1}{2}}$ if $p(A)$ is not the largest element.
    \item $\mathbf{j}\left(A\right)=\infty$ if $|A|<N_{\mathbf{n}\left(A\right)}$; $|c\left(p(A)\right)|=\min \{|A|,N_{\mathbf{n}\left(A\right)}\}$.
    \item $\delta_2\left(A\right)=2^{\frac{\mathbf{n}\left(A\right)}{2}}r^{-\mathbf{j}\left(A\right)}$. If $A_1, A_2 \in c\left(A\right)$, then $ (A_1+\delta_2\left(A_1\right)B_2)\bigcap (A_2+\delta_2\left(A_2\right)B_2) = \varnothing$.
\end{enumerate}
and similar to (\ref{3.24-1}), we define
$$\tau^{(2)} \left(\mathcal{T}^{(2)}\right) = \inf\limits_{t\in S_{\mathcal{T}^{(2)}}}\sum\limits_{t\in A, A\in  \cup_{i\ge 0}\mathcal{B}_i}2^{\frac{\mathbf{n}(A)}{2}}\delta_2\left(A\right).$$

\begin{remark}
Note that the definition of $\mathcal{T}^{(2)}$ differs slightly from our previously defined iterative organized tree. Under the specific condition $p=2$, due to (\ref{case_p_2}), we may concretize the abstract components in the original definition the iterative organized tree. For example, instead of the inclusion relation $$B_{\mathbf{j}\left(A\right)}\left(2^{\mathbf{n}\left(A\right)}\right)\subset B_{\mathbf{j}\left(p\left(A\right)\right)}\left(2^{\mathbf{n}\left(p\left(A\right)\right)}\right),$$
we employ $$r^{\mathbf{j}(A)-\mathbf{j}(p(A))}\geq 2^{\frac{\mathbf{n}(A)-\mathbf{n}(p(A))+1}{2}}.$$
Consequently, there is no essential difference between the structure of $\mathcal{T}^{(2)}$ and that of the iterative organized tree.
\end{remark}

\begin{theorem}\label{them 3.3}
We have
 $$\sup \{\tau^{(2)} \left(\mathcal{T}^{(2)}\right); \mathcal{T}^{(2)} \subset T\}\lesssim_r\textsf{E} \sup _{t \in T} X_t.$$  
 
\end{theorem}

\begin{proof}
Note that
$$r^{\mathbf{j}(A)-\mathbf{j}(p(A))}\geq 2^{\frac{\mathbf{n}(A)-\mathbf{n}(p(A))+1}{2}}$$ implies 
$$
B_{\mathbf{j}\left(A\right)}\left(2^{\mathbf{n}\left(A\right)}\right)\subset r^{-\mathbf{j}(A)}2^{\frac{\mathbf{n}(A)+1}{2}} \subset r^{-\mathbf{j}(p(A))}2^{\frac{\mathbf{n}(p(A))}{2}} \subset B_{\mathbf{j}\left(p\left(A\right)\right)}\left(2^{\mathbf{n}\left(p\left(A\right)\right)}\right).
$$
$B\left(2^n\right)\subset 2^{\frac{n+1}{2}}B_2$ and $ (A_1+\delta_2\left(A_1\right)B_2)\bigcap (A_2+\delta_2\left(A_2\right)B_2) = \varnothing$ imply $$\left(\sqrt{2}A_1+B_{\mathbf{j}\left(A_1\right)}\left(2^{\mathbf{n}\left(A_1\right)}\right)\right)\bigcap \left(\sqrt{2}A_2+B_{\mathbf{j}\left(A_2\right)}\left(2^{\mathbf{n}\left(A_2\right)}\right)\right) = \varnothing.$$ 
Thus any tree $\mathcal{T}^{(2)}$ is a parameterized separation tree of $\sqrt{2}T$. So we have 
\begin{align*}
\sup \{\tau^{(2)} \left(\mathcal{T}^{(2)}\right); \mathcal{T}^{(2)} \subset T\}&\leq \sup \{\tau \left(\mathcal{T}\right); \mathcal{T} \text{ parameterized separation tree }\subset \sqrt{2}T\}
\\
&\lesssim_r \textsf{E} \sup _{t \in \sqrt{2}T} X_t = \sqrt{2}\textsf{E} \sup _{t \in T} X_t.
\end{align*}
\end{proof}

   We can define iterative organized trees $\Gamma^{(2)} _{\left(\mathbf{n},\mathbf{j},J\right)}\left(T\right)$ in this case. In this definition, we need to replace $B_{\mathbf{j}\left(t\right)}\left(2^{n_m^{\left(k+1\right)}}\right)\subset B_{\mathbf{j}\left(A_m\right)}\left(2^{\mathbf{n}\left(A_m\right)}\right)$ with $r^{\mathbf{j}(t)-\mathbf{j}(A_m)}\geq 2^{\frac{n_m^{\left(k+1\right)}-\mathbf{n}\left(A_m\right)+1}{2}}$. For brevity, we refrain from providing explicit definitions herein.
   
   As before, we may define the following. 
\begin{align*}
    &\rho^{(2)}\left(\Gamma^{(2)} _{\left(\mathbf{n},\mathbf{j},J\right)}\left(T\right)\right)=\inf _{t \in S_{\Gamma^{(2)} _{\left(\mathbf{n},\mathbf{j},J\right)}\left(T\right)}}\sum_{t\in A}\log \left(|c\left(p\left(A\right)\right)|\right)r^{-\mathbf{j}\left(A\right)} ;\\
&F^{(2)}_{n,j}\left(T\right)=\sup \{\rho\left(\Gamma^{(2)} _{\left(\mathbf{n},\mathbf{j},J\right)}\left(T\right)\right); n_{0}^{\left(0\right)}\geq n, J=j+1+\frac{1}{\kappa}(n_{0}^{\left(0\right)}-n)\} \text{ for } j\geq \tilde{j}_0;\\
&Size^{(2)}\left(T\right)=F^{(2)}_{0,\tilde{j}_0}\left(T\right)=\sup \{\rho\left(\Gamma^{(2)} _{\left(\mathbf{n},\mathbf{j},J\right)}\left(T\right)\right); n_{0}^{\left(0\right)}\geq n, J\geq \tilde{j}_0\}.
\end{align*}

\begin{theorem}\label{them 3.4}
    We have
    $$Size^{(2)}\left(T\right)\gtrsim_r\textsf{E} \sup _{t \in T} X_t.$$
\end{theorem}

\begin{proof}
    As in the proof of Theorem \ref{them 2.5}, we need to prove that $F^{(2)}_{n, j}$ satisfies the growth condition under conditions (\ref{1}) and (\ref{2}) (see Definition \ref{growth conditon}). This proof is almost the same as the proof of Proposition \ref{prop 3}. The only difference from the proof of Proposition \ref{prop 3} is that we no longer need to show
\[
B_{\mathbf{j}(A)}\!\left(2^{\mathbf{n}(A)}\right)\subset B_{j+1}\!\left(2^{n}\right)
\quad \text{for all } A\in\mathcal{A}_1.
\]
Instead, it suffices to verify that
\[
r^{\mathbf{j}(A)-(j+1)} \ge 2^{\frac{\mathbf{n}(A)-n+1}{2}} .
\]
 By $\mathbf{j}(A)\geq j+3+\frac{1}{\kappa}(\mathbf{n}(A)-n-\kappa)$ and $r= 2^{\kappa-3}$, we have $$r^{\mathbf{j}(A)-(j+1)}\geq 2^\frac{(\kappa-3)(\mathbf{n}(A)-n+1)}{\kappa}\geq 2^{\frac{\mathbf{n}(A)-n+1}{2}}\text{ for }\kappa\geq 6.$$
\end{proof}

\begin{theorem}\label{them 3.5}
One has
    $$ \sup \{\tau^{(2)} \left(\mathcal{T}^{(2)}\right); \mathcal{T}^{(2)} \subset T\}\sim_r\textsf{E} \sup _{t \in T} X_t.$$
\end{theorem}

\begin{proof}
    Firstly, the second inequality is Theorem \ref{them 3.3}. Furthermore, by $\sqrt{2^n}B_2\subset B\left(2^n\right)$, we have $$\left(A_1+B_{\mathbf{j}\left(A_1\right)}\left(2^{\mathbf{n}\left(A_1\right)}\right)\right)\bigcap \left(A_2+B_{\mathbf{j}\left(A_2\right)}\left(2^{\mathbf{n}\left(A_2\right)}\right)\right) = \varnothing,$$ which implies $$ (A_1+\delta_2\left(A_1\right)B_2)\bigcap (A_2+\delta_2\left(A_2\right)B_2) = \varnothing.$$ Thus,  when $J\geq \tilde{j}_0$, we can say that any $\Gamma^{(2)} _{\left(\mathbf{n},\mathbf{j},J\right)}\left(T\right)$ is a $\mathcal{T}^{(2)}$ as in the proof of Theorem \ref{them 3.1}. So we have $$\sup \{\tau^{(2)} \left(\mathcal{T}^{(2)}\right); \mathcal{T}^{(2)} \subset T\} = C(r) Size^{(2)}(T)\gtrsim_r\textsf{E} \sup _{t \in T} X_t.$$
\end{proof}

\subsection{Case 2: $p=1$.}
In this case, for $|x|\ge 1$ we have
\[
|x| \le \hat U_i(x)=2|x|-1 \le x^2 .
\]
Then we have
\[
\hat U_i(x)\le x^2 \quad \text{and} \quad \hat U_i(x)\le 2|x|
\qquad \text{for all } x\in\mathbb{R}.
\]
Hence,
\[
\sum_{i\ge1} a_i^2 \le u \;\Rightarrow\; \sum_{i\ge1} \hat U_i(a_i)\le u,
\]
and
\[
\sum_{i\ge1} 2|a_i| \le u \;\Rightarrow\; \sum_{i\ge1} \hat U_i(a_i)\le u .
\]

Consequently, we have $\mathcal{N}_u\left(t\right) \geq \sqrt{u}\|t\|_2$ and $\mathcal{N}_u\left(t\right) \geq u\|t\|_{\infty} / 2$. Moreover, if $\sum_{i \geq 1} \hat{U}_i\left(a_i\right) \leq u$, writing $b_i=a_i \mathbf{1}_{\left\{\left|a_i\right| \geq 1\right\}}$ and $c_i=a_i \mathbf{1}_{\left\{\left|a_i\right|<1\right\}}$ we have $\sum_{i \geq 1}\left|b_i\right| \leq$ $u$ (since $\hat{U}_i\left(x\right) \geq|x|$ for $|x| \geq 1$) and $\sum_{i \geq 1} c_i^2 \leq u$ (since $\hat{U}_i\left(x\right) \geq x^2$ for $|x| \leq 1$). So we have
$$
\sum_{i \geq 1} t_i a_i=\sum_{i \geq 1} t_i b_i+\sum_{i \geq 1} t_i c_i \leq u\|t\|_{\infty}+\sqrt{u}\|t\|_2.
$$
And we have shown that
$$
\frac{1}{2} \max \left(u\|t\|_{\infty}, \sqrt{u}\|t\|_2\right) \leq \mathcal{N}_u\left(t\right) \leq\left(u\|t\|_{\infty}+\sqrt{u}\|t\|_2\right),
$$
which implies
$$
\frac{1}{2}\left(B_{\infty}\left(0,1\right) \cap B_2\left(0, \sqrt{u}\right)\right) \subset B\left(u\right) \subset 2\left(B_{\infty}\left(0,1\right) \cap B_2\left(0, \sqrt{u}\right)\right) .
$$

Now we can define a tree $\mathcal{T}^{(1)}$ (In this example, $r\geq 2^7$):
\begin{enumerate}
    \item $\mathbf{j}(A)=J$, $\mathbf{n}(A)=-1$ if $A$ is the largest element of $\mathcal{T}^{(1)}$.
    \item $\mathbf{n}\left(A\right)>\mathbf{n}\left(p\left(A\right)\right)$.
    \item $\mathbf{j}\left(A\right)\geq \Tilde{j}_0$ if $p(A)$ is the largest element.
    \item $r^{\mathbf{j}(A)-\mathbf{j}(p(A))}\geq 2^{\frac{\mathbf{n}(A)-\mathbf{n}(p(A))+4}{2}}$ if $p(A)$ is not the largest element.
    \item $\mathbf{j}\left(A\right)=\infty$ if $|A|<N_{\mathbf{n}\left(A\right)}$; $|c\left(p(A)\right)|=\min \{|A|,N_{\mathbf{n}\left(A\right)}\}$.
    \item $\delta_2\left(A\right)=2^{\frac{\mathbf{n}\left(A\right)}{2}}r^{-\mathbf{j}\left(A\right)}$;$\delta_\infty\left(A\right)=r^{-\mathbf{j}\left(A\right)}$
    \item If $A_1, A_2 \in c\left(A\right)$, then 
    $$\left(A_1+r^{-\mathbf{j}(A_1)}\left(B_{\infty} \cap 2^{\frac{\mathbf{n}(A_1)}{2}}B_2\right)\right)\bigcap \left(A_2+r^{-\mathbf{j}(A_2)}\left(B_{\infty} \cap 2^{\frac{\mathbf{n}(A_2)}{2}}B_2\right)\right) = \varnothing. $$

\end{enumerate}
Similarly to (\ref{3.24-1}), we define  $$\tau^{(1)} \left(\mathcal{T}^{(1)}\right) = \inf\limits_{t\in S_{\mathcal{T}^{(1)}}}\sum\limits_{t\in A, A\in  \cup_{i\ge 0}\mathcal{B}_i}2^{\frac{\mathbf{n}(A)}{2}}\delta_2\left(A\right) =\inf\limits_{t\in S_{\mathcal{T}^{(1)}}}\sum\limits_{t\in A, A\in  \cup_{i\ge 0}\mathcal{B}_i} 2^{\mathbf{n}(A)}\delta_{\infty}\left(A\right).$$

\begin{theorem}\label{them 5.4}
We have
        $$\textsf{E} \sup _{t \in T} X_t\leq \sup \{\tau^{(1)} \left(\mathcal{T}^{(1)}\right); \mathcal{T}^{(1)} \subset T\}\sim_r \textsf{E} \sup _{t \in T} X_t.$$
\end{theorem}

\begin{proof}
    The proof is almost the same as the proof in Theorem \ref{them 3.5}. Here we do not provide specific details, but only point out a few steps that correspond to the previous proof, but different.
   
     Firstly, 
        $$r^{\mathbf{j}(A)-\mathbf{j}(p(A))}\geq 2^{\frac{\mathbf{n}(A)-\mathbf{n}(p(A))+4}{2}}$$ implies 
        $$
            \mathbf{j}(A)>\mathbf{j}(p(A))\text{ and }r^{-\mathbf{j}(p(A))}2^{\frac{\mathbf{n}(p(A))}{2}}\geq 4r^{-\mathbf{j}(A)}2^{\frac{\mathbf{n}(A)}{2}}.
        $$
        Thus
        \begin{align*}
            4r^{-\mathbf{j}(A)}B_{\infty}\left(0,1\right)\subset r^{-\mathbf{j}(p(A))}B_{\infty}\left(0,1\right);\\
            4r^{-\mathbf{j}(A)}2^{\frac{\mathbf{n}(A)}{2}}B_2\subset r^{-\mathbf{j}(p(A))}2^{\frac{\mathbf{n}(p(A))}{2}}B_2.
        \end{align*}
        Then, we have
        \begin{align*}
        B_{\mathbf{j}\left(A\right)}\left(2^{\mathbf{n}\left(A\right)}\right)\subset 2r^{-\mathbf{j}(A)}\left(B_{\infty}\left(0,1\right) \cap B_2\left(0, \sqrt{2^{\mathbf{n}(A)}}\right)\right)\\
        \subset \frac{1}{2}r^{-\mathbf{j}(p(A))}\left(B_{\infty}\left(0,1\right) \cap B_2\left(0, \sqrt{2^{\mathbf{n}(p(A))}}\right)\right)\subset B_{\mathbf{j}\left(p\left(A\right)\right)}\left(2^{\mathbf{n}\left(p\left(A\right)\right)}\right). 
        \end{align*}
Additionally, $$\left(A_1+r^{-\mathbf{j}(A_1)}\left(B_{\infty} \cap 2^{\frac{\mathbf{n}(A_1)}{2}}B_2\right)\right)\bigcap \left(A_2+r^{-\mathbf{j}(A_2)}\left(B_{\infty} \cap 2^{\frac{\mathbf{n}(A_2)}{2}}B_2\right)\right) =\varnothing$$ implies
        $$
        \left(2A_1+B_{\mathbf{j}\left(A_1\right)}\left(2^{\mathbf{n}\left(A_1\right)}\right)\right)\bigcap \left(2A_2+B_{\mathbf{j}\left(A_2\right)}\left(2^{\mathbf{n}\left(A_2\right)}\right)\right) =\varnothing.
        $$
We obtain 
$$ \sup \{\tau^{(1)} \left(\mathcal{T}^{(1)}\right); \mathcal{T}^{(1)} \subset T\}\lesssim_r \sup _{t \in T} X_t.$$

Furthermore, we can define a new tree with $r^{\mathbf{j}(t)-\mathbf{j}(A_m)}\geq 2^{\frac{n_m^{\left(k+1\right)}-\mathbf{n}(A_m)+4}{2}}$. 
We need to prove $r^{\mathbf{j}(A)-(j+1)}\geq 2^{\frac{\mathbf{n}(A)-n+4}{2}}$ by $\mathbf{j}(A)\geq j+3+\frac{1}{\kappa}(\mathbf{n}(A)-n-\kappa)$ and $\mathbf{n}(A)\geq \kappa$.   It is easy to see that
        \begin{align*}
            r^{\mathbf{j}(A)-(j+1)}\geq 2^\frac{(\kappa-3)(\mathbf{n}(A)-n+1)}{\kappa}\geq 2^{\frac{\mathbf{n}(A)-n+4}{2}} \text{ for } \kappa\geq 10.
        \end{align*}
Thus, we only need to assume that $r$ is sufficiently large.
In fact, 
        $$\left(\frac{1}{2}A_1+B_{\mathbf{j}\left(A_1\right)}\left(2^{\mathbf{n}\left(A_1\right)}\right)\right)\bigcap \left(\frac{1}{2}A_2+B_{\mathbf{j}\left(A_2\right)}\left(2^{\mathbf{n}\left(A_2\right)}\right)\right) = \varnothing$$ implies
       $$\left(A_1+r^{-\mathbf{j}(A_1)}\left(B_{\infty} \cap 2^{\frac{\mathbf{n}(A_1)}{2}}B_2\right)\right)\bigcap \left(A_2+r^{-\mathbf{j}(A_2)}\left(B_{\infty} \cap 2^{\frac{\mathbf{n}(A_2)}{2}}B_2\right)\right) = \varnothing. $$
        And we can consider the new trees in $\frac{1}{2}T$.
   
 Thus we obtain $$\textsf{E} \sup _{t \in T} X_t\lesssim_r \sup \{\tau^{(1)} \left(\mathcal{T}^{(1)}\right); \mathcal{T}^{(1)} \subset T\}.$$
\end{proof}

\begin{remark}
In fact, the two trees in the above examples can be described in an alternative form, whose structure depends only on the function $\mathbf{n}$.

For $p=2$, the new tree $\mathcal{T}^{(2)\prime}, (r\ge 8)$ can be described as:
\begin{enumerate}
    \item $\mathbf{j}(A)=J$, $\mathbf{n}(A)=-1$ if $A$ is the largest element of $\mathcal{T}^{(2)\prime}$.
    \item $\mathbf{n}\left(A\right)>\mathbf{n}\left(p\left(A\right)\right)$. 
    \item $|c\left(p(A)\right)|=\min \{|A|,N_{\mathbf{n}\left(A\right)}\}$.
    \item If $A_1 \in c\left(A\right)$ and $A$ is the largest element, then
    \[
    \mathbf{j}(A_1)=
    \begin{cases}
    +\infty &\text{ if } |c\left(p(A)\right)|<N_{\mathbf{n}\left(A\right)},\\
      \begin{aligned}
      \inf_{j\in \mathbb{Z}}\{&(A_1+\delta_2\left(A_1\right)B_2)\bigcap (A_2+\delta_2\left(A_2\right)B_2) = \varnothing \\
      &\text{ for }\forall A_2\ne A_1\in c(A) \text{ and } j\geq\Tilde{j}_0\}
      \end{aligned} &\text{ if } |c\left(p(A)\right)|=N_{\mathbf{n}\left(A\right)}.
    \end{cases}
    \]
    \item If $A_1 \in c\left(A\right)$ and $A$ is not the largest element, then
    \[
    \mathbf{j}(A_1)=
    \begin{cases}
    +\infty &\text{ if } |c\left(p(A)\right)|<N_{\mathbf{n}\left(A\right)},\\
      \begin{aligned}
      \inf_{j\in \mathbb{Z}}&\{(A_1+\delta_2\left(A_1\right)B_2)\bigcap (A_2+\delta_2\left(A_2\right)B_2) = \varnothing \\
      &\text{ for }\forall A_2\ne A_1\in c(A) \text{ and } r^{j}\geq r^{\mathbf{j}(A)}2^{\frac{\mathbf{n}(A_1)-\mathbf{n}(A)+1}{2}}\}
      \end{aligned} &\text{ if } |c\left(p(A)\right)|=N_{\mathbf{n}\left(A\right)}.
    \end{cases}
    \]
\end{enumerate}    
then, we can define $\tau^{(2)\prime} \left(\mathcal{T}^{(2)\prime}\right) = \inf\limits_{t\in S_{\mathcal{T}^{(2)\prime}}}\sum\limits_{t\in A\in \mathcal{T}^{(2)\prime}}2^{\frac{\mathbf{n}(A)}{2}}\delta_2\left(A\right)$.

Since $\mathcal{T}^{(2)\prime}$ and $\mathcal{T}^{(2)}$ have no essential differences, we have
    \begin{equation}\label{eq:equv_2}
     \sup \{\tau^{(2)\prime} \left(\mathcal{T}^{(2)\prime}\right); \mathcal{T}^{(2)\prime} \subset T\}\sim_r\textsf{E} \sup _{t \in T} X_t.
    \end{equation}

Similarly, we have a new tree $\mathcal{T}^{(1)\prime}$($r\geq 2^7$):
\begin{enumerate}
    \item $\mathbf{j}(A)=\infty$, $\mathbf{n}(A)=-1$ if $A$ is the largest element of $\mathcal{T}^{(2)\prime}$.
    \item $\mathbf{n}\left(A\right)>\mathbf{n}\left(p\left(A\right)\right)$. 
    \item $|c\left(p(A)\right)|=\min \{|A|,N_{\mathbf{n}\left(A\right)}\}$.
    \item If $A_1 \in c\left(A\right)$ and $A$ is the largest element, then
\[
\mathbf{j}(A_1)=
\begin{cases}
+\infty &\text{ if } |c\left(p(A)\right)|<N_{\mathbf{n}\left(A\right)},\\
\begin{aligned}
\inf_{j\in \mathbb{Z}}\{&\left(A_1+r^{-\mathbf{j}(A_1)}\left(B_{\infty} \cap 2^{\frac{\mathbf{n}(A_1)}{2}}B_2\right)\right)\bigcap\\& \left(A_2+r^{-\mathbf{j}(A_2)}\left(B_{\infty} \cap 2^{\frac{\mathbf{n}(A_2)}{2}}B_2\right)\right) = \varnothing\\
      & \text{ for }\forall A_2\ne A_1\in c(A) \text{ and } j\geq\Tilde{j}_0
\}
\end{aligned} &\text{ if } |c\left(p(A)\right)|=N_{\mathbf{n}\left(A\right)}.
\end{cases}
\]
    \item If $A_1 \in c\left(A\right)$ and $A$ is not the largest element, then
\[
\mathbf{j}(A_1)=
\begin{cases}
+\infty &\text{ if } |c\left(p(A)\right)|<N_{\mathbf{n}\left(A\right)},\\
\begin{aligned}
\inf_{j\in \mathbb{Z}}\{&\left(A_1+r^{-\mathbf{j}(A_1)}\left(B_{\infty} \cap 2^{\frac{\mathbf{n}(A_1)}{2}}B_2\right)\right)\bigcap\\& \left(A_2+r^{-\mathbf{j}(A_2)}\left(B_{\infty} \cap 2^{\frac{\mathbf{n}(A_2)}{2}}B_2\right)\right) = \varnothing\\
      & \text{ for }\forall A_2\ne A_1\in c(A) \text{ and } r^{j}\geq r^{\mathbf{j}(A)}2^{\frac{\mathbf{n}(A_1)-\mathbf{n}(A)+4}{2}}
\}
\end{aligned} &\text{ if } |c\left(p(A)\right)|=N_{\mathbf{n}\left(A\right)}.
\end{cases}
\]
\end{enumerate}
Similarly, we have$$
\tau^{(1)\prime} \left(\mathcal{T}^{(1)\prime}\right) = \inf\limits_{t\in S_{\mathcal{T}^{(1)\prime}}}\sum\limits_{t\in A\in \mathcal{T}^{(1)\prime}}2^{\frac{\mathbf{n}(A)}{2}}\delta_2\left(A\right) =\inf\limits_{t\in S_{\mathcal{T}^{(1)\prime}}}\sum\limits_{t\in A\in \mathcal{T}^{(1)\prime}} 2^{\mathbf{n}(A)}\delta_{\infty}\left(A\right)
$$
and\begin{equation}\label{eq:equv_1}
   \sup \{\tau^{(1)\prime} \left(\mathcal{T}^{(1)\prime}\right); \mathcal{T}^{(1)\prime} \subset T\}\sim_r\textsf{E} \sup _{t \in T} X_t. 
\end{equation}

The relations \eqref{eq:equv_2} and \eqref{eq:equv_1} are easily reached through $\mathcal{T}^{(1)}$ and $\mathcal{T}^{(2)}$, hence we do not provide proof.

\end{remark}

\section{Application}

In this section, we present a deterministic algorithm for computing an approximation to $\textsf{E} \sup\limits_{t\in T}X_t$ as an application of our main result.  We assume that $|T|<\infty$. Thus, we can let $N(T)=|T|$ and $n(T)$ be the smallest numbers such that $|T|\leq N_{n(T)}$. We can also let $j(T)$ be the largest number such that $\varphi_{j(T)}(x_i,x_{i^\prime})>2^{n(T)}$ for $\forall x_i\neq x_{i^\prime}\in T$. Obviously $j(T)$ can be computed. Then we will get a deterministic algorithm for computing an approximation to $\textsf{E} \sup\limits_{t\in T}X_t$ in time poly($N(T)$).

To present our result, we introduce the growth condition for points, which plays an essential role in the proof. Before that, we need to give some notions.

Put
$$
k_n(A)=\sup\limits_{|S|\leq N_{n+\kappa}}\sup\{j\in \mathbb{Z}: A\subset \bigcup\limits_{x\in S}B_j(x,2^{n})\},
$$
$$
j_n(A)=\sup\{j\in \mathbb{Z}: \varphi_j(x,y)\leq 2^n \text{ for } \forall x,y \in A\},
$$
$$
\gamma(T)=\inf\limits_{\mathcal{A}}\sup\limits_{t\in T}\sum\limits_{n\geq 0}2^nr^{-j_n\left(A_n\left(t\right)\right)},
$$
where $\mathcal{A}$ means any admissible sequence, $\kappa>0$.

\begin{remark}
    Obviously, $k_n(A)\geq j_n(A)$ for any $n$ and $A$.
\end{remark}

As we did in Lemma \ref{gc lemma 1}, we can prove the following lemma.
\begin{lemma}\label{lemma for point}
We have $k_n(A)=+\infty$ or for $\forall \delta\geq 1$ there exist $(x_i)_{i\leq N_{n+\kappa}} \in A$ such that 
$$
\varphi_{k_n(A)+\delta}(x_i,x_j)>2^{n} \text{ for } \forall i\neq j.
$$
\end{lemma}

Thus, similarly, we have the following theorem.
\begin{theorem}[Contraction principle for points]\label{Contraction Principle for point}
Let $a\geq 0$ satisfy
$$
r^{-k_n(A)}\leq C_3\left( ar^{-j_n(A)}+\sup\limits_{x\in A}r^{-s_n(x)} \right)
$$
for all $n\geq 0$ and $A\subset T$. Then
$$
\gamma(T)\leq Poly(C_3,r) \left( a\gamma(T) + \sup\limits_{x\in T}\sum\limits_{n\geq 0}2^nr^{-s_n(x)} + r^{-j_0(T)}\right).
$$

\end{theorem}
\begin{proof}
Assume that $r^{-s_n(x)}\leq r^{-j_0(T)}$ for all $n\geq 0$ and $x\in T$. Let $(\mathcal{A}_n)$ be an admissible sequence of $T$. For every $n\geq 1$ and partition element $A_n\in \mathcal{A}_n$, we construct sets $A_n^{ij}$ as follows. We first partition $A_n$ into $n$ segments (here $1\leq i<n$):
$$
A_n^i=\{x\in A_n;2^{-2i}r^{-j_0(T)}<r^{-s_n(x)}\leq 2^{-2(i-1)}r^{-j_0(T)}\}.
$$
$$
A_n^n=\{x\in A_n;r^{-s_n(x)}\leq 2^{-2(n-1)}r^{-j_0(T)}\}.
$$
Using the assumption of the theorem, we can partition each $A_n^i$ into at most $N_{n+\kappa}$ pieces $A_n^{ij}$ such that 
$$
j_{n+3+\kappa}(A_n^{ij})\geq j_n(A_n^{ij})\geq k_n(A_n^{i})-1,
$$
$$
r^{-k_n(A_n^{ij})}\leq r^{-k_n(A_n^i)}\leq C_3\left(ar^{-j_n(A_n)}+r^{-s_n(x)}+2^{-(n-1)}r^{-j_0(T)}\right)
$$
for all $x\in A_n^{ij}$. Let $\mathcal{C}_{n+3+\kappa}$ be the partition generated by all sets $A_k^{ij},k\leq n,i,j$ thus constructed. Then $$|\mathcal{C}_{n+3+\kappa}|\leq\prod_{k=1}^{n}k(2^{2^k})(2^{2^{k+\kappa}})<N_{n+3+\kappa}.$$

Define $\mathcal{C}_k=\{T\}$ for $0\leq k \leq 3+\kappa,$ we obtain
\begin{align*}
    \gamma(T)&\leq \sup\limits_{x\in T} \sum_{n\geq 0} 2^nr^{-j_n(C_n(t))}\\
    &\leq Poly(C_3,r)\left( a\sup\limits_{t\in T}\sum\limits_{n\geq0}2^nr^{-A_n(t)} + \sup\limits_{x\in T}\sum_{n\geq 0}2^{n}r^{-s_n(x)} + r^{-j_0(T)}\right),
\end{align*}
where the universal constant only depends on $r$. 
\end{proof}

Now, we introduce the growth condition for the points.
\begin{definition}\label{growth condition point}
    We say that $f_{n,j}$ satisfy the $\sigma$-growth condition for points if the following occurs. Recall Lemma \ref{lemma 2.3},
    We can assume  
    $r=2^{\kappa-3}$, where $\kappa$ is a large enough integer such that \eqref{eq:lemma 2.3} holds for $\rho<\frac{1}{10}$. Moreover, we assume $\kappa\ge 5$.
    Let functions $f_{n, j}$ be maps from $T$ to $\mathbb{R}^{+}$ for $n \in \mathbb{N}, j \in \mathbb{Z}$, satisfying 
    $$f_{n,j}\geq \max\{f_{n+\kappa,j+1},f_{n,j+1}\}$$ 
    and 
    $$
    \begin{aligned}
    f_{n,j}=\sup\{f_{n,j} : n\geq 0,j\geq j_0(T)-\sigma-2\} \text{ for all } j\leq j_0(T)-\sigma-3 \text{ and } n\geq 0 
    \end{aligned}
    $$
    for some integer $\sigma\ge 0$. Consider any points $(t_i)_{i\leq N_{n+\kappa}} \in l_2$, and $j$ such that
    $$
    \begin{aligned}
    \varphi_{j+2} (x_i/8,x_{i^\prime}/8)>2^{n+\kappa} \text{ for } \forall i\neq i^\prime,   \\
    \exists u, \text{ s.t. } x_i \in B_{j-\sigma}(u,2^n) \text{ for } \forall i.
    \end{aligned}
    $$
    Then 
    $$
    f_{n,j-\sigma-1}(u)\geq 2^nr^{-j} + \min\limits_{i\leq N_{n+\kappa}}f_{n+\kappa,j+2}(t_i).
    $$
\end{definition}

Put
    $$
    K_n(t,x)=\inf_{s\in \mathbb{Z}}\{tr^{-s}+\sup_{u\in T,n,j}f_{n,j}(u)-\sup_{u\in B_s\left(x,2^n\right)}f_{n,s}\left(u\right)\},
    $$
and set
\begin{align*}
K_n(c_4a2^n,x)&\leq c_4a2^nr^{-s_n^a(x)} + \sup_{u\in T,n,j}f_{n,j}(u)-\sup_{u\in B_{s_n^{a}(x)}\left(x,2^n\right)}f_{n,s_n^a(x)}\left(u\right)\\
&\leq K_n(c_4a2^n,x) + 2^{-n}\sup_{u\in T,n,j}f_{n,j}(u).
\end{align*}

\begin{lemma}\label{interpolation for point}
We have
$$
\sup\limits_{t\in T}\sum_{n\geq 0}2^nr^{-s_n^a(x)}\lesssim \frac{1}{r}\left(\frac{\sup_{u\in T,n,j}f_{n,j}(u)}{a}+r^{-j_0(T)}\right).
$$
\end{lemma}

\begin{proof}
Obviously, $0\leq K_{n}\leq \sup_{u\in T,n,j}f_{n,j}(u)$ for all $n\geq 0$.
In addition, it is easy to see $s_n^a(x)\geq j_0(T)-\sigma-3$ for all $n\geq 0$. As we did in Lemma \ref{interpolation},
$$
    \begin{aligned}
    2^{-n}&\sup_{u\in T,n,j}f_{n,j}(u)+K_{n+\kappa}(c_4a2^{n+\kappa},x)-K_n(c_4a2^n,x)\\
    &\geq c_4a(2^\kappa-r)2^nr^{-s_{n+\kappa}^a(x)}+\sup_{u\in B_{s_{n+\kappa}^a(x)-1}(x,2^{n})}f_{n,s_{n+\kappa}^a(x)-1}(u)
    \\&\quad\quad-\sup_{u\in B_{s_{n+\kappa}^a(x)}(x,2^{n+\kappa})}f_{n+\kappa,s_{n+\kappa}^a(x)}(u)\\
    &\geq c_4a(8r-r)2^nr^{-s_{n+\kappa}^a(x)}+\sup_{u\in B_{s_{n+\kappa}^a(x)-1}(x,2^{n})}f_{n+\kappa,s_{n+\kappa}^a(x)}(u)
    \\&\quad\quad-\sup_{u\in \rho B_{s_{n+\kappa}^a(x)-1}(x,2^{n})}f_{n+\kappa,s_{n+\kappa}^a(x)}(u)\\
    &\gtrsim ar2^nr^{-s_{n+\kappa}^a(x)}.
    \end{aligned}
 $$ 

We conclude by summing up these inequalities for all $n\geq 0$.
\end{proof}

\begin{theorem}\label{growth theorem point}
    Let $\sigma$ be large enough. If functions $f_{n,j}$ satisfy $\sigma$- growth condition for points, then we have
    $$
    \gamma(T)\lesssim_r\left(\sup_{u\in T,n,j}f_{n,j}(u)+ r^{-j_0(T)}\right).
    $$
\end{theorem}

\begin{proof}
    Set
    $$
    \Delta_n(x)=K_{n+\kappa}(c_4a2^{n+\kappa},x)-K_n(c_4a2^n,x),
    $$
    $$
    r^{-s_n(x)}=2^{-n}\left(\Delta_n(x)+2^{-n}\sup_{u\in T,n,j}f_{n,j}(u)+2^{-n}r^{-j_0(T)}\right)+Cr^{-s_{n+\kappa}^a(x)} \text{ for } n\geq 0.
    $$
    We want to prove $$r^{-k_n(A)}\leq Poly(r)\left(\left(aC_5+\frac{1}{C_5}\right)r^{-j_n(A)}+\sup\limits_{x\in A}r^{-s_{n}(x)}\right)$$ for all $n\geq 0$ and $A\subset T$. Thus, we can assume $k_n(A)<+\infty$ and $r^{-j_n(A)}\leq C_5r^{-k_n(A)}$. Here, $C_5$ is actually $r^{\sigma-1}$. 
    
    $n=0$ is obvious. Let $n\geq 1$. By Lemma \ref{lemma for point} there exist $(x_i)_{i\leq N_{n+\kappa}} \in A$ such that 
    $$
    \varphi_{k_n(A)+1}(x_i,x_j)> 2^{n} \text{ for } \forall i\neq j.
    $$

    Define $R = \sup\limits_{x\in A}r^{-s_{n+\kappa}^a(x)}$. If $R > r^{-(k_n(A)+2)}$, then the conclusion is automatically satisfied.

    If $R \leq r^{-(k_n(A)+2)}$, then $s_{n+\kappa}^a(x)\geq k_n(A)+2$ for all $x\in A$. Let
    $$
    \delta = \min\{j_n(A),\min\limits_{x\in A}s_{n+\kappa}^a(x)\}-1=j_n(A)-1.
    $$
    Then we have 
    $$
    r^{-j_n(A)}+R\geq r^{-(\delta+1)}\geq \frac{r^{-j_n(A)}+R}{2}
    $$
    and
    $$
    r^{-\delta}=r^{-j_n(A)+1}\leq r^{-k_n(A)+\sigma}.
    $$
    Set $H_i=B_{s_{n+\kappa}^a(x_i)}(x_i,2^{n+\kappa})\subset  B_{k_n(A)+2}(2^{n+\kappa})\subset B_{k_n(A)+1}(2^{n})$.
    
      For $\forall y \in H_j$ and $\forall i \leq N_{n+\kappa}$, 
    \begin{align*}
        y-x_i&= y-x_j+x_j-x_i\\
        &\in B_{j_n(A)+1}(2^n)+B_{j_n(A)}(2^n)\\
        &\subset B_{j_n(A)-1}(2^n)\subset B_\delta(2^n).
    \end{align*}
    Thus, $$\bigcup\limits_{i\leq N_{n+\kappa}}H_i\subset B_\delta(x_i,2^{n})\subset B_{k_n(A)-\sigma}(x_i,2^{n})$$ for all $i\leq N_{n+\kappa}$.
    For $\forall y_i\in H_i$ and $\forall y_j\in H_j$,
    $$
    \varphi_{k_n(A)+2}(y_i/8,y_j/8)> 2^{n+\kappa}.
    $$
    Otherwise
    $$
    \begin{aligned}
    x_i-x_j&=(y_j-x_j)-(y_i-x_i)+(y_j-y_i)\\
    &\in 10B_{k_n(A)+2}(2^{n+\kappa})\subset 10\rho B_{k_n(A)+1}(2^n)   \\
    &\subset B_{k_n(A)+1}(2^n),
    \end{aligned}
    $$
    which contradicts $
    \varphi_{k_n(A)+1}(x_i,x_j)> 2^{n} $.
    
      Then
    \begin{align*}
        2^nr^{-k_n(A)}
        &\leq f_{n,k_n(A)-\sigma-1}(x_j)-\sup_{u\in H_j}f_{n+\kappa,k_n(A)+2}(u) \text{ for some } j \leq N_{n+\kappa}\\
        &\leq \sup_{u\in B_{k_n(A)-\sigma-1}(x_j,2^{n})}f_{n,k_n(A)-\sigma-1}(u)-\sup_{u\in B_{s_{n+\kappa}^a(x_j)}(x_j,2^{n+\kappa})}f_{n+\kappa,s_{n+\kappa}^a(x_j)}(u)\\
        &\leq c_4a2^{n}r^{-j_n(A)+\sigma+1}+\sup_{u\in T,n,j}f_{n,j}(u)-K_n(c_4a2^n,x_j)+K_{n+\kappa}(c_4a2^{n+\kappa},x_j)\\
        &-\sup_{u\in T,n,j}f_{n,j}(u)-c_4a2^{n+\kappa}r^{-s_{n+\kappa}^a(x_j)}+2^{-n+1}\sup_{u\in T,n,j}f_{n,j}(u),
    \end{align*}
    which implies 
    \begin{align*}
    &r^{-k_n(A)}\\
    &\leq Poly(r)\left(aC_5r^{-j_n(A)}+\sup\limits_{x\in A}2^{-n}\left(\Delta_n(x)+2^{-n}\sup_{u\in T,n,j}f_{n,j}(u)\right)+\sup\limits_{x\in A}Car^{-s_n^a(x)}\right)\\
    &\leq Poly(r)\left(aC_5r^{-j_n(A)}+\sup\limits_{x\in A}r^{-s_n(x)}\right).
    \end{align*}
    here we can assume $a<1$.

    Combining Theorem \ref{Contraction Principle for point}, Lemma \ref{interpolation for point}, and $r^{-j_0(T)}\leq Poly(r)$, the conclusion follows since we can choose suitable numbers $a$ and $\sigma$.
\end{proof}

Our goal is to construct a family of functions $\{f_{n,j}\}_{n\in \mathbb{N},j\in \mathbb{Z}}$ on $T$ that satisfy the $\sigma$- growth condition for points. 
We first define the functions $f_{n,j}=0$ in the parameter grid:
\[
\begin{cases} 
n(T)+\kappa< n, & j_0-\sigma-3<j,\\
0\leq n, &j(T) + \sigma + 3 \leq j.
\end{cases}
\]

Subsequently, we extend the definition to all \( (n,j) \in \{n\in \mathbb{N}: n\leq n(T)+\kappa\} \times \{j\in \mathbb{Z}:j_0-\sigma-2\leq j \leq j(T)+\sigma+3\} \) via the recursive scheme when $f_{n+z,j+z^\prime}$ has been assigned to any $z, z^\prime\in \mathbb{N}$ and $z+z^\prime \neq 0$. This procedure constitutes the desired algorithm.

Before we determine $f_{n,j}$, we can define a function $A_{n,j}(u) : T \to \mathbb{R}$ for any fixed $u\in T$.

Firstly, we can find $x_1\in B_{j+1}(u,2^n)$ such that 
$$f_{n+\kappa,j+\sigma+3}(x_1)=\sup_{x\in B_{j+1}(u,2^n)}f_{n+\kappa,j+\sigma+3}(x).$$ For $1\leq m<N_{n+\kappa}$, when $(x_i)_{i\leq m}$ has been defined, if 
$$B_{j+1}(u,2^n)/\bigcup_{i\leq m}4B_{j+\sigma+3}(x_i,2^{n+\kappa})=\varnothing,$$ then we stop and let $A_{n,j}(u)=0$. Otherwise, we can find 
$$x_{m+1}=\sup\left\{f_{n+\kappa,j+\sigma+3}(x) : x\in B_{j+1}(u,2^n)/\bigcup_{i\leq m}4B_{j+\sigma+3}(x_i,2^{n+\kappa})\right\}.$$ 
When $(x_i)_{i\leq N_{n+\kappa}}$ is defined, we set 
$$
A_{n,j}(u)=2^{n}r^{-j-\sigma-1}+f_{n+\kappa,j+\sigma+3}(x_{N_{n+\kappa}}).
$$

Now we define 
\begin{equation}
    f_{n,j}(u)=\max\left\{f_{n+\kappa,j+1}(u),A_{n,j}(u),f_{n,j+1}(u)\right\} \label{iteration}
\end{equation}
for all $u\in T$. 

Finally, we define $f_{n,j}$ for all $j\leq j_0-\sigma-3$ and $n\geq 0$, which  satisfies the growth condition for the points.  Thus, we get a family of functions $\{f_{n,j}\}_{n\in \mathbb{N},j\in \mathbb{Z}}$ on $T$.

Now, we present our result on the deterministic algorithm for computing an approximation to $\textsf{E} \sup\limits_{t\in T}X_t$ in time poly($N(T)$).
\begin{theorem}
We have
    $$
    \gamma(T)\lesssim_r\left(\sup_{u\in T,n,j}f_{n,j}(u)+ r^{-j_0(T)}\right).
    $$
\end{theorem}
\begin{proof}
    By Theorem \ref{growth theorem point},  we need to prove that $\{f_{n,j}\}$ satisfies the growth condition for points.
    
        Obviously $f_{n,j}\geq 0$, $$f_{n,j}\geq \max\{f_{n+\kappa,j+1},f_{n,j+1}\}$$ and $$f_{n,j}=\sup_{n\geq 0,j\geq j_0(T)-\sigma-2}f_{n,j}$$ for all $j\leq j_0(T)-\sigma-3$ and $n\geq 0$.

    If there exist points $(t_i)_{i\leq N_{n+\kappa}} \in l_2$, and $j$ such that
    $$
    \begin{aligned}
    \varphi_{j+\sigma+3} (t_i/8,t_{i^\prime}/8)>2^{n+\kappa} \text{ for } \forall i\neq i^\prime;   \\
    \exists u, \text{ s.t. } t_i \in B_{j+1}(u,2^n) \text{ for } \forall i,
    \end{aligned}
    $$
    then $j+\sigma+3\geq j_0+1$, $j+1\leq j(T)$ and $n\leq n(T)$. This implies $j_0-\sigma-2\leq j\leq j(T)-1$ and $n\leq n(T)$. Thus, we can find $(x_i)_{i\leq N_{n+\kappa}}\in B_{j+1}(u,2^n)$ such that 
    \begin{align*}
    &\qquad f_{n,j}(u)
    \\&\geq 2^{n}r^{-j-\sigma-1}+\sup\left\{f_{n+\kappa,j+\sigma+3}(x_i) : x\in B_{j+1}(u,2^n)/\bigcup_{i< N_{n+\kappa}}4B_{j+\sigma+3}(x_i,2^{n+\kappa})\right\}.    
    \end{align*}
    
   By the separation condition
\[
\varphi_{j+\sigma+3}(t_i/8,t_{i'}/8) > 2^{n+\kappa}
\quad \text{for } i\neq i',
\]
each ball $4B_{j+\sigma+3}(x_i,2^{n+\kappa})$ contains at most one point of
$(t_i)_{i\le N_{n+\kappa}}$.
 Thus there exists at least one point $t_i\in B_{j+1}(u,2^n)/\bigcup_{i<N_{n+\kappa}}4B_{j+\sigma+3}(x_i,2^{n+\kappa})$ such that
    \begin{align*}
     &\min\limits_{i\leq N_{n+\kappa}}f_{n+\kappa,j+\sigma+3}(t_i)
     \\&\leq \sup\left\{f_{n+\kappa,j+\sigma+3}(x_i) : x\in B_{j+1}(u,2^n)/\bigcup_{i< N_{n+\kappa}}4B_{j+\sigma+3}(x_i,2^{n+\kappa})\right\}.   
    \end{align*}
    Thus we have
    $$
    f_{n,j-\sigma-1}(u)\geq 2^nr^{-j-\sigma-1} + \min\limits_{i\leq N_{n+\kappa}}f_{n+\kappa,j+\sigma+3}(t_i).
    $$
\end{proof}

\begin{theorem}
We have
    $$
    \sup_{u\in T,n,j}f_{n,j}(u)+r^{-j_0}\lesssim_r \textsf{E}\sup\limits_{t\in T}X_t.
    $$
\end{theorem}

\begin{proof}
    We need to prove $\sup_{u\in T,n,j}f_{n,j}(u)\lesssim_r \textsf{E}\sup\limits_{t\in T}X_t$. Note that there exists $(n,j,u_0)$ such that $$f_{n,j}(u_0)=\sup_{u\in T,n,j}f_{n,j}(u).$$
    Define $$S_0=\left\{(n,j);f_{n,j}(u_0)=\sup_{u\in T,n,j}f_{n,j}(u)\right\}.$$
    
      Then we can find $(n^0,j^0)\in S_0$ such that $n^0\geq n$ and $j^0\geq j$ for $\forall (n,j)\in S_0$ by (\ref{iteration}). Thus, it is easy to see $f_{n^0,j^0}(u_0)=A_{n^0,j^0}(u_0)$, which implies that there exist points $(x_i^0)_{i\leq N_{n^0+\kappa}}\in B_{j+1}(u_0,2^n)$ such that
    $$
    \begin{aligned}
        \Bar{\varphi}_{j^0+\sigma+3}(x_i^0,x_{i^\prime}^0)>2^{n^0+\kappa} \text{ for } \forall i\neq i^\prime \leq N_{n^0+\kappa},\\
        f_{n^0,j^0}(u_0)=2^{n^0}r^{-j^0-\sigma-1}+\min_{i\leq N_{n^0+\kappa}}f_{n^0+\kappa,j^0+\sigma+3}(x_i^0).\\
    \end{aligned}
    $$
    We then set
\[
\mathcal{B}_0 = \left\{
A_i^0 := B_{j^0+\sigma+3}(x_i^0, 2^{n^0+\kappa}) \cap T \;:\; i \le N_{n^0+\kappa}
\right\},
\]
and set
\[
\mathbf{j}(A_i^0) = j^0 + \sigma + 3, \quad
\mathbf{n}(A_i^0) = n^0 + \kappa.
\]

    For iteration, If we have $\mathcal{B}_k$ for $k\geq 0$ and any set $A_m$ that belongs to $\mathcal{B}_k$ have been defined as $B_{j^k}(x_m^k,2^{n^k})\cap T$. If $f_{n^k,j^k}(x_m^k)=0$, then stop; otherwise we can find the largest number $j^{k\prime}$ and the largest number $n^{k\prime}$ such that
    $$
    \begin{aligned}
        j^{k\prime}-j^k\geq \frac{1}{\kappa}\left(n^{k\prime}-n^k\right)\in \mathbb{N},\\
        f_{n^k,j^k}(x_m^k)=f_{n^{k\prime},j^{k\prime}}(x_m^k),\\
        f_{n^{k\prime},j^{k\prime}}(x_m^k)=A_{n^{k\prime},j^{k\prime}}(x_m^k).
    \end{aligned}
    $$
    As discussed previously, we can find $(x_i^{k+1})_{i\leq N_{n^{k\prime}+\kappa}}\in B_{j^{k\prime}+1}(x_m^k,2^{n^{k\prime}})$ such that

    \begin{align*}
    \Bar{\varphi}_{j^{k\prime}+\sigma+3}(x_i^{k+1},x_{i^\prime}^{k+1})>2^{n^{k\prime}+\kappa} \text{ for } \forall i\neq i^\prime \leq N_{n^{k\prime}+\kappa},\\
    f_{n^{k\prime},j^{k\prime}}(x_m^k)=2^{n^{k\prime}}r^{-j^{k\prime}-\sigma-1}+\min_{i\leq N_{n^{k\prime}+\kappa}}f_{n^{k\prime}+\kappa,j^{k\prime}+\sigma+3}(x_i^{k+1}),
     \end{align*}
    and define 
    $$
    \begin{aligned}
        \left\{A_i^{k+1}=B_{j^{k\prime}+\sigma+3}(x_i^{k+1},2^{n^{k\prime}+\kappa})\cap T;i\leq N_{n^{k\prime}+\kappa}\right\}\subset \mathcal{B}_{k+1},\\
        \mathbf{n}(A_i^{k+1})=n^{k\prime}+\kappa; \mathbf{j}(A_i^{k+1})=j^{k\prime}+\sigma+3.
    \end{aligned}
    $$
        For any $i$, if $y\in A_i^k$, then we have 
    $$
    \begin{aligned}
        y-x_m^k&=y-x_i^{k+1}+x_i^{k+1}-x_m^k\\
        &\in B_{j^{k\prime}+\sigma+3}(2^{n^{k\prime}+\kappa})+B_{j^{k\prime}+1}(2^{n^{k\prime}})\\
        &\subset 2r^{-1}B_{j^{k\prime}+1}(2^{n^{k\prime}})\subset B_{j^{k\prime}}(2^{n^{k\prime}})\\
        &\subset (\rho r)^{\frac{1}{\kappa}(n^{k\prime}-n^k)}B_{j^{k\prime}}(2^{n^{k\prime}}) \subset B_{j^{k}}(2^{n^{k\prime}}).
    \end{aligned}
    $$
    Then we obtain $A_i^k\subset A_m$ for $\forall i$ and $B_{j^{k\prime}+\sigma+3}(2^{n^{k\prime}+\kappa})\subset B_{j^{k}}(2^{n^{k\prime}})$.
    
      It is easy to see $j^0+\sigma+3\geq \tilde{j}_0$. Thus we get an organized iterative tree $\Gamma _{\left(\mathbf{n}, \mathbf{j}\right)}\left(T\right)$. By construction, it is easy to see $f_{n^0,j^0}(u_0)=\frac{2^{-\kappa}r^{2}}{\log2}\rho\left(\Gamma _{\left(\mathbf{n}, \mathbf{j}\right)}\left(T\right)\right)$. Thus, we have
    $$
    \sup_{u\in T,n,j}f_{n,j}(u)\lesssim_rSize(T) \lesssim_rE\sup\limits_{t\in T}X_t.
    $$

    \end{proof}

In practice, since the computation of $A_{n,j}$ is based on the points in $T$, we need not perform computations for all $j$. This is the key to decoupling the algorithm's time complexity from $j(T)$.

\begin{theorem}
    We can compute $$\sup\Bigg\{f_{n,j} : 0\leq n\leq n(T)+\kappa,j_0-\sigma-2\leq j\leq j(T)+\sigma+3\Bigg\}$$ in polynomial time in $N(T)$.
\end{theorem}

\begin{proof}
Note that there exist points $x_1, x_2\in B_{j+1}(u,2^n)$ such that $x_1-x_2\notin 4B_{j+\sigma+3}(2^{n+\kappa})$ when $A_{n,j}(u)\neq 0$. It follows that
$$
x_1\notin B_{j+\sigma+2}(u,2^n) \text{ or } x_2\notin B_{j+\sigma+2}(u,2^n).
$$
Otherwise, $x_1-u\in B_{j+\sigma+2}(u,2^n) $ and $x_2-u\in B_{j+\sigma+2}(u,2^n)$, which imply $x_1-x_2\in B_{j+\sigma+2}(u,2^n)\subset 4B_{j+\sigma+3}(2^{n+\kappa})$.

Now we can define a set,
$$
\begin{aligned}
&I=\{j\in \mathbb{Z}: \exists x,y\in T\text{ and natural number }n\leq n(T)+\kappa \text{ such that } \\ &x-y\notin B_{j+\sigma+2}(2^n) \text{ and }x-y\in B_{j+1}(2^n)\}\cap \left\{j:j_0-\sigma-2\leq j \leq j(T)+\sigma+3\right\}.
\end{aligned}
$$
Then we can easily construct $I$ in time Poly($N(T)$) by $$I\subset \bigcup_{0\leq n\leq n(T)+\kappa}\bigcup_{x,y\in T}\{j\in \mathbb{Z}:x-y\notin B_{j+\sigma+2}(2^n) \text{ and }x-y\in B_{j+1}(2^n)\}.$$
Furthermore, we have $\left|I\right|\leq Poly(N(T))$. Thus,  we conclude that $j\in I$ if $A_{n,j}(u)\neq 0$.

We now focus on constructing $f_{n,j}$ in the domain $I$.  We denote the modified version by $g_{n,j}$.

We first define the functions $g_{n,j}$ in the parameter grid:
\[
\begin{cases} 
n(T)+\kappa< n, & j_0-\sigma-3<j\in I,\\
0\leq n, &j(T) + \sigma + 3 \leq j\in I.
\end{cases}
\]
Subsequently, we extend the definition to all \( (n,j) \in \{n\in \mathbb{N}: n\leq n(T)+\kappa\} \times I\) via the recursive scheme when $g_{n+z\kappa,j+z^\prime}$ have been assigned for any $z, z^\prime\in \mathbb{N}$, $z\leq z^\prime\neq 0$ and $j,j+z^\prime\in I$.

Similarly, we define a new function $B_{n,j}(u) : T \to \mathbb{R}$ for any fixed $u\in T$. Before that, for notational simplicity, we define a function $$h_{n+\kappa,j+\sigma+3}=\sup\Bigg\{g_{n+\kappa+z\kappa,j+\sigma+3+z^\prime} : z, z^\prime\in \mathbb{N}\text{, } z\leq z^\prime \text{ and } j+\sigma+3+z^\prime\in I\Bigg\}.$$

First, we can find $x_1\in B_{j+1}(u,2^n)$ such that 
$$h_{n+\kappa,j+\sigma+3}(x_1)=\sup_{x\in B_{j+1}(u,2^n)}h_{n+\kappa,j+\sigma+3}(x).$$ for $1\leq m<N_{n+\kappa}$, when $(x_i)_{i\leq m}$ have been defined, if 
$$B_{j+1}(u,2^n)/\bigcup_{i\leq m}4B_{j+\sigma+3}(x_i,2^{n+\kappa})=\varnothing,$$ then we stop and let $B_{n,j}(u)=0$, otherwise we can find 
$$x_{m+1}=\sup\left\{h_{n+\kappa,j+\sigma+3}(x) : x\in B_{j+1}(u,2^n)/\bigcup_{i\leq m}4B_{j+\sigma+3}(x_i,2^{n+\kappa})\right\}.$$ 
When $(x_i)_{i\leq N_{n+\kappa}}$ are defined, we set
$$
B_{n,j}(u)=2^{n}r^{-j-\sigma-1}+h_{n+\kappa,j+\sigma+3}(x_{N_{n+\kappa}}).
$$

Now we define 
\begin{equation}
    g_{n,j}(u)=\max\Bigg\{\sup\Big\{g_{n+z\kappa,j+z^\prime}: z, z^\prime\in \mathbb{N}\text{, } z\leq z^\prime \text{ and } j+z^\prime\in I\Big\},B_{n,j}(u)\Bigg\} \label{iteration}
\end{equation}
for all $u\in T$. 

Finally, we get $g_{n,j}$ over \( (n,j) \in \{n\in \mathbb{N}: n\leq n(T)+\kappa\} \times I\) in time polynomial in $N(T)$. 

By construction, it is easy to see that $$g_{n,j}=\sup\left\{g_{n+z\kappa,j+z^\prime}: z, z^\prime\in \mathbb{N}\text{, } z\leq z^\prime \text{ and } j+z^\prime\in I\right\}$$ when $j\in I$. Thus, we can extend this definition to $$ (n,j) \in \{n\in \mathbb{N}: n\leq n(T)+\kappa\} \times \{j\in \mathbb{Z}:j_0-\sigma-2\leq j \leq j(T)+\sigma+3\} $$ by defining $$g_{n,j}=\sup\left\{g_{n+z\kappa,j+z^\prime}: z, z^\prime\in \mathbb{N}\text{, } z\leq z^\prime \text{ and } j+z^\prime\in I\right\}.$$

Next, we need to prove that $f_{n,j}\equiv g_{n,j}$. We can prove this by induction.

Firstly $f_{n,j}\equiv g_{n,j}$ on the parameter grid:
\[
\begin{cases} 
n(T)+\kappa< n, & j_0-\sigma-3<j\in I,\\
0\leq n, &j(T) + \sigma + 3 \leq j\in I.
\end{cases}
\]
Furthermore, 
we can assume that $f_{n+z,j+z^\prime}\equiv g_{n+z,j+z^\prime}$ for any $z, z^\prime\in \mathbb{N}$ and $z+z^\prime \neq 0$. If $j\neq I$, we have $$f_{n,j}=\sup_{z, z^\prime\in \mathbb{N}\text{, } z\leq z^\prime\neq 0} f_{n+z\kappa,j+z^\prime}=\sup_{z, z^\prime\in \mathbb{N}\text{, } z\leq z^\prime\neq 0} g_{n+z\kappa,j+z^\prime}=g_{n,j}.$$
If $j\in I$, we get $A_{n,j}=B_{n,j}$ by constructions and assumptions. Thus
$$
\begin{aligned}
f_{n,j}&=\max\{\sup_{z, z^\prime\in \mathbb{N}\text{, } z\leq z^\prime\neq 0} f_{n+z\kappa,j+z^\prime},A_{n,j}\}
\\&=\max\{\sup_{z, z^\prime\in \mathbb{N}\text{, } z\leq z^\prime\neq 0} g_{n+z\kappa,j+z^\prime},B_{n,j}\}=g_{n,j}.  
\end{aligned}
$$
By induction, we derive $f_{n,j}\equiv g_{n,j}$ over $$ (n,j) \in \{n\in \mathbb{N}: n\leq n(T)+\kappa\} \times \{j\in \mathbb{Z}:j_0-\sigma-2\leq j \leq j(T)+\sigma+3\}.$$

Therefore, we can compute  $$\sup\left\{f_{n,j} : 0\leq n\leq n(T)+\kappa,j_0-\sigma-2\leq j\leq j(T)+\sigma+3\right\}$$ by computing $\sup\left\{g_{n,j} : 0\leq n\leq n(T)+\kappa,j\in I\right\}$ in polynomial time in $N(T)$.

\end{proof}


\textbf{Acknowledgment}.The work of Hanchao Wang (corresponding author) was supported by the National Key R\&D Program of China (No.2024YFA1013501), the National Natural Science Foundation of China (No. 12571162), and Shandong Provincial Natural Science Foundation (No. ZR2024MA082). The work of Vladimir V. Ulyanov  was supported by the HSE University Basic Research Program and the program of the Moscow Center for Fundamental and Applied Mathematics, Lomonosov Moscow State University.
\printbibliography

\end{document}